\documentclass[twocolumn,10pt]{asme2ej}

\usepackage{epsfig} %% for loading postscript figures
\usepackage{amsmath,amssymb,url}
\usepackage{graphicx,subfigure}
\usepackage{color}

\newcommand{\norm}[1]{\ensuremath{\left\| #1 \right\|}}

\newcommand{\bracket}[1]{\ensuremath{\left[ #1 \right]}}
\newcommand{\braces}[1]{\ensuremath{\left\{ #1 \right\}}}

\newcommand{\refeqn}[1]{(\ref{eqn:#1})}
\newcommand{\reffig}[1]{Figure \ref{fig:#1}}
\newcommand{\tr}[1]{\mbox{tr}\ensuremath{\negthickspace\bracket{#1}}}
\newcommand{\trs}[1]{\mathrm{tr}\ensuremath{[#1]}}

\newcommand{\SO}{\ensuremath{\mathsf{SO(3)}}}
\newcommand{\T}{\ensuremath{\mathsf{T}}}

\newcommand{\SE}{\ensuremath{\mathsf{SE(3)}}}

\renewcommand{\Re}{\ensuremath{\mathbb{R}}}
\newcommand{\Sph}{\ensuremath{\mathsf{S}}}

\newcommand{\D}{\ensuremath{\mathbf{D}}}

\usepackage{dsfont}
\usepackage{hyperref}

\newtheorem{prop}{Proposition}

\title{Geometric Adaptive Tracking Control of a Quadrotor UAV on $\SE$ for Agile Maneuvers}

\author{Farhad A. Goodarzi
    \affiliation{
	Ph.D. Candidate\\
	MAE Department\\
	The George Washington University\\
	Washington, DC 20052\\
    Email: fgoodarzi@gwu.edu
    }	
}

\author{Daewon Lee
    \affiliation{ Post Doctoral Fellow\\
	MAE Department\\
	The George Washington University\\
	Washington, DC 20052\\
        Email: daewonlee@gwu.edu
    }
}

\author{Taeyoung Lee
    \affiliation{ Assistant Professor\\
	Flight Dynamics and Control Laboratory\\
	MAE Department\\
	The George Washington University\\
	Washington, DC 20052\\
        Email: tylee@gwu.edu
    }
}
\begin{document}

\maketitle    

%%%%%%%%%%%%%%%%%%%%%%%%%%%%%%%%%%%%%%%%%%%%%%%%%%%%%%%%%%%%%%%%%%%%%%
\begin{abstract}
{\it This paper presents nonlinear tracking control systems for a quadrotor unmanned aerial vehicle under the influence of uncertainties. Assuming that there exist unstructured disturbances in the translational dynamics and the attitude dynamics, a geometric nonlinear adaptive controller is developed directly on the special Euclidean group. In particular, a new form of an adaptive control term is proposed to guarantee stability while compensating the effects of uncertainties in quadrotor dynamics. A rigorous mathematical stability proof is given. The desirable features are illustrated by numerical example and experimental results of aggressive maneuvers.
}
\end{abstract}

%%%%%%%%%%%%%%%%%%%%%%%%%%%%%%%%%%%%%%%%%%%%%%%%%%%%%%%%%%%%%%%%%%%%%%
\begin{nomenclature}

\entry{$\vec{e}_{i}\in\Re^3$} {Inertial frame}
\entry{$\vec{b}_{i}\in\Re^3$} {Body-fixed frame}
\entry{$m\in\Re$} {Mass of the quadrotor}
\entry{$J\in\Re^{3\times 3}$} {Inertia matrix of the quadrotor}
\entry{$R\in\Re^{3\times 3}$} {Rotation matrix(body-fixed to inertial frame)}
\entry{$\Omega\in\Re^3$} {Angular velocity with respect to body fixed frame}
\entry{$x\in\Re^3$} {Quadrotor position}
\entry{$v\in\Re^3$} {Quadrotor velocity}
\entry{$f\in\Re$} {Total thrust}
\entry{$M\in\Re^3$} {Total moment}
\entry{$g\in\Re$} {Gravitational acceleration}
\entry{$R_d\in\Re^{3\times 3}$} {Desired rotation matrix}
\entry{$\Omega_d\in\Re^3$} {Desired angular velocity}
\entry{$\Psi\in\Re$} {Attitude error function}
\entry{$\SO$} {Special Orthogonal  group}
\entry{$\SE$} {Special Euclidean group}
\end{nomenclature}

%%%%%%%%%%%%%%%%%%%%%%%%%%%%%%%%%%%%%%%%%%%%%%%%%%%%%%%%%%%%%%%%%%%%%%
\section{Introduction}

Quadrotor unmanned aerial vehicles (UAVs) are becoming increasingly popular. They offer flight characteristics comparable to traditional helicopters, namely stationary, vertical, and lateral flights in a wide range of speeds, with a much simpler mechanical structure. With their small-diameter rotors driven by electric motors, these multi-rotor platforms are safer to operate than helicopters in indoor environments. Also, they have sufficient payload transporting capability and flight endurance for various missions~\cite{Mahony2012,gooddaewontaeyoungacc14}.  

Several control systems have been proposed for quadrotors. In many cases, disturbances and uncertainties are eliminated in the model for simplicity. There are other limitations of quadrotor control systems, such as complexities in controller structures or lack of stability proof. For example, tracking control of a quadrotor UAV has been considered in~\cite{CabCunPICDC09,MelKumPICRA11}, but the control system in~\cite{CabCunPICDC09} has a complex structure since it is based on a multiple-loop backstepping approach, and no stability proof is presented in~\cite{MelKumPICRA11}. Robust tracking control systems are studied in~\cite{NalMarPICDC09,HuaHamITAC09}, but the quadrotor dynamics is simplified by considering planar motion only~\cite{NalMarPICDC09}, or by ignoring the rotational dynamics by timescale separation assumption~\cite{HuaHamITAC09}. 

In other studies, disturbances and uncertainties have been considered into the dynamics of the quadrotors~\cite{NormanKam2013,HassanFaiz2013,Besnard2007,Bolandi2013}. Several controllers have been designed and presented to eliminate these disturbances such as PID~\cite{Sharma2012}, sliding mode~\cite{Liu2013}, or robust controllers~\cite{Wahyudie2013}. In one example, proportional-derivative controllers are developed with consideration of blade flapping for operations under wind disturbances~\cite{HofHuaAGNCC07}. A backstepping control method is proposed in~\cite{Castillo2006} by considering an aggressive perturbation with bounded signals. These approaches have certain limitations on handling uncertainties. For example, it is well known that sliding mode controller causes chattering problems that may excite high-frequency unmodeled dynamics. Nonlinear robust tracking control systems in~\cite{LeeLeoPACC12,LeeLeoAJC13} guarantee ultimate boundedness of tracking errors only, and they are also prone to chattering if the required ultimate bound is smaller. PID controllers are adopted widely, but it is required that the uncertainties are fixed. 

%Most of these approaches are not completely suitable for tracking trajectories using a quadrotor UAV, such as sliding mode controller which causes chattering problems in the response or PID controllers which usually need offline gain tuning and usage of different optimization methods.

Due to the limitations mentioned above, adaptive controllers have been developed and are very popular for tracking desired trajectories in existence of disturbances and uncertainties~\cite{dydekAnnnaLav12,Justin2014,Antonelli2013,HZhenchina13}. Although adaptive controllers are very robust for quadrotors trajectory tracking in existence of disturbances and uncertainties, most of the studies are based on linearization~\cite{dydekAnnnaLav12,Justin2014} or simplification~\cite{Antonelli2013,HZhenchina13}. In~\cite{Antonelli2013} only the constant external disturbances is considered into the system dynamics and the stability analysis. In~\cite{HZhenchina13} an adaptive block backstepping controller is presented to stabilize the attitude of a quadrotor, however this method only guarantees the boundedness of errors. A nonlinear adaptive state feedback controller is also presented in~\cite{Silvestre2014}, where the proposed controller only assumes constant known disturbance forces. An adaptive sliding mode controller is developed for under-actuated quadrotor dynamics in~\cite{DWLee2009}. This controller uses slack variables to overcome the under-actuated property of a quadrotor system while simplifying the dynamics to reduce the higher-order derivative terms which makes it very sensitive to the noise. A robust adaptive control of a quadrotor is also presented in~\cite{Bialy2013}, where linear-in-the-parameter uncertainties and bounded disturbances are considered. There is lack of numerical and experimental validations in this study to show the robustness or capability of running aggressive maneuvers using the proposed controller. These simplifications~\cite{Antonelli2013,Silvestre2014}, linearization~\cite{dydekAnnnaLav12,Justin2014}, and assumptions~\cite{HZhenchina13} in the dynamics and controller design process of adaptive controllers restrict the quadrotor to maintain complex or aggressive missions such as a flipping maneuver~\cite{Bialy2013}. 

The other critical issue in designing controllers for quadrotors is that they are mostly based on local coordinates. Some aggressive maneuvers are demonstrated at~\cite{MelMicIJRR12} which are based on Euler angles. Therefore they involve complicated expressions for trigonometric functions, and they exhibit singularities in representing quadrotor attitudes, thereby restricting their ability to achieve complex rotational maneuvers significantly. A quaternion-based feedback controller for attitude stabilization was shown in~\cite{TayMcGITCSTI06}. By considering the Coriolis and gyroscopic torques explicitly, this controller guarantees exponential stability. Quaternions do not have singularities but, as the three-sphere double-covers the special orthogonal group, one attitude may be represented by two antipodal points on the three-sphere. This ambiguity should be carefully resolved in quaternion-based attitude control systems, otherwise they may exhibit unwinding, where a rigid body unnecessarily rotates through a large angle even if the initial attitude error is small~\cite{BhaBerSCL00}. To avoid these, an additional mechanism to lift attitude onto the unit-quaternion space is introduced~\cite{MaySanITAC11}.

Recently, the dynamics of a quadrotor UAV is globally expressed on the special Euclidean group, $\SE$, and nonlinear control systems are developed to track outputs of several flight modes~\cite{LeeLeoPICDC10}. Several aggressive maneuvers of a quadrotor UAV are demonstrated based on a hybrid control architecture, and a nonlinear robust control system is also considered in~\cite{LeeLeoPACC12,Farhad2013}. As they are directly developed on the special Euclidean group, complexities, singularities, and ambiguities associated with minimal attitude representations or quaternions are completely avoided~\cite{ChaSanICSM11}.

This paper is an extension of the prior work of the authors in~\cite{LeeLeoPICDC10,Farhad2013,LeeLeoAJC13}. Geometric nonlinear controllers are developed to follow an attitude tracking command and a position tracking command. In particular, a new form of an adaptive control term is proposed to guarantee asymptotical convergence of tracking error variables when there exist uncertainties at the translational dynamics and the rotational dynamics of quadrotors where the disturbances are considered arbitrary without any simplification. The corresponding stability properties are analyzed mathematically, and it is verified by several experiments. This is significantly in contrast to the existing various experimental results for a quadrotor UAV where control systems are applied \textit{ad hoc} without careful stability analyses. The robustness of the proposed tracking control systems are critical in generating complex maneuvers, as the impact of the several aerodynamic effects resulting from the variation in air speed is significant even at moderate velocities~\cite{HofHuaAGNCC07}.

In short, new contributions and the unique features of the control system proposed in this paper compared with other studies are as follows: (i) it is developed for the full six degrees of freedom dynamic model of a quadrotor UAV on $\SE$, including the coupling effects between the translational dynamics and the rotational dynamics on a nonlinear manifold without any simplification or assumptions, (ii) the control systems are developed directly on the nonlinear configuration manifold in a coordinate-free fashion. This yields remarkably compact expressions for the dynamic model and controllers, compared with local coordinates that often require symbolic computational tools due to complexity of multi-body systems. Thus, singularities of local parameterization are completely avoided to generate agile maneuvers in a uniform way, (iii) a rigorous Lyapunov analysis is presented to establish stability properties without any timescale separation assumption, and (iv) a new form of an adaptive control term is proposed to guarantee asymptotical convergence of tracking error variables when there exist uncertainties at the translational dynamics and the rotational dynamics of quadrotors where the disturbances are considered arbitrary without any simplification, (v) in contrast to hybrid control systems~\cite{GilHofIJRR11}, complicated reachability set analysis is not required to guarantee safe switching between different flight modes, as the region of attraction for each flight mode covers the configuration space almost globally, (vi) the proposed algorithm is validated with experiments for agile maneuvers. To the author's best knowledge, a rigorous mathematical analysis of nonlinear adaptive controllers of a quadrotor UAV on $\SE$ with experimental validations for complex and aggressive maneuvers is unprecedented. 
%and guaranteed to be robust against unstructured uncertainties in both the translational dynamics and the rotational dynamics where the disturbances are considered arbitrary without any simplification,

The paper is organized as follows. We develop a globally defined model for a quadrotor UAV in Section \ref{sec:QDM}. A hybrid control architecture is introduced and an adaptive attitude tracking control system is developed in Section \ref{sec:ACFM}. Section \ref{sec:PCFM} present results for an adaptive position tracking, followed by numerical examples in Section \ref{sec:NE}. Finally, Section \ref{sec:ER} presents two experimental results of aggressive maneuvers. 
%%%%%%%%%%%%%%%%%%%%%%%%%%%%%%%%%%%%%%%%%%%%%%%%%%%%%%%%%%%%%%%%%%%%%%
\section{QUADROTOR DYNAMICS MODEL}\label{sec:QDM}
Consider a quadrotor UAV model illustrated in \reffig{QM}. We choose an inertial reference frame $\{\vec e_1,\vec e_2,\vec e_3\}$ and a body-fixed frame $\{\vec b_1,\vec b_2,\vec b_3\}$. The origin of the body-fixed frame is located at the center of mass of this vehicle. The first and the second axes of the body-fixed frame, $\vec b_1,\vec b_2$, lie in the plane defined by the centers of the four rotors.

The configuration of this quadrotor UAV is defined by the location of the center of mass and the attitude with respect to the inertial frame. Therefore, the configuration manifold is the special Euclidean group $\SE$, which is the semi-direct product of $\Re^3$ and the special orthogonal group $\SO=\{R\in\Re^{3\times 3}\,|\, R^TR=I,\, \det{R}=1\}$. 

The mass and the inertial matrix of a quadrotor UAV are denoted by $m\in\Re$ and $J\in\Re^{3\times 3}$. Its attitude, angular velocity, position, and velocity are defined by $R\in\SO$, $\Omega,x,v\in\Re^3$, respectively, where the rotation matrix $R$ represents the linear transformation of a vector from the body-fixed frame to the inertial frame and the angular velocity $\Omega$ is represented with respect to the body-fixed frame. The distance between the center of mass to the center of each rotor is $d\in\Re$, and the $i$-th rotor generates a thrust $f_i$ and a reaction torque $\tau_i$ along $-\vec b_3$ for $1\leq i \leq 4$. The magnitude of the total thrust and the total moment in the body-fixed frame are denoted by $f\in\Re$, $M\in\Re^3$, respectively. 
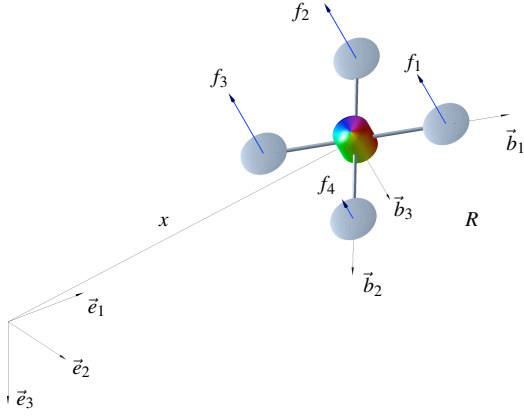
\begin{figure}
\setlength{\unitlength}{0.8\columnwidth}\footnotesize
\centerline{
\begin{picture}(1,0.8)(0,0)
\put(0,0){\includegraphics[width=0.8\columnwidth]{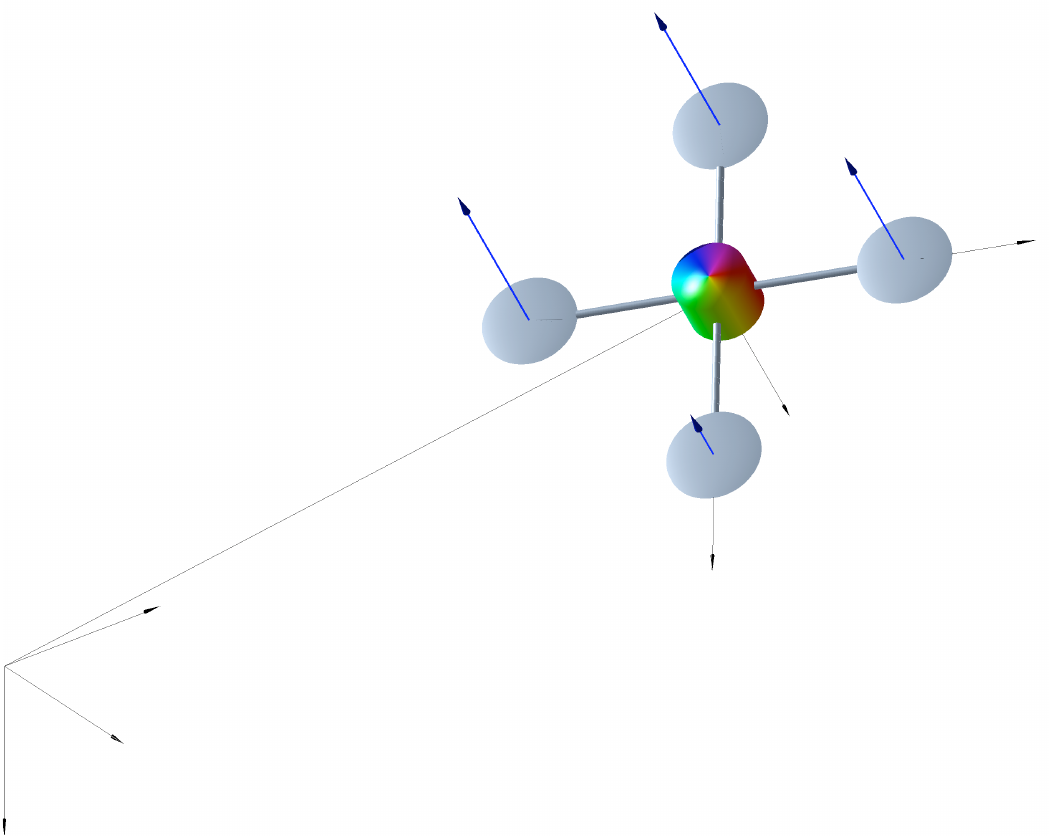}}
\put(0.16,0.18){\shortstack[c]{$\vec e_1$}}
\put(0.13,0.06){\shortstack[c]{$\vec e_2$}}
\put(0.02,0.0){\shortstack[c]{$\vec e_3$}}
\put(0.98,0.5){\shortstack[c]{$\vec b_1$}}
\put(0.70,0.22){\shortstack[c]{$\vec b_2$}}
\put(0.76,0.37){\shortstack[c]{$\vec b_3$}}
\put(0.78,0.66){\shortstack[c]{$f_1$}}
\put(0.56,0.76){\shortstack[c]{$f_2$}}
\put(0.40,0.63){\shortstack[c]{$f_3$}}
\put(0.61,0.42){\shortstack[c]{$f_4$}}
\put(0.30,0.35){\shortstack[c]{$x$}}
\put(0.90,0.35){\shortstack[c]{$R$}}
\end{picture}}
\caption{Quadrotor model}\label{fig:QM}
\end{figure}
The following conventions are assumed for the rotors and propellers, and the thrust and moment that they exert on the quadrotor UAV. We assume that the thrust of each propeller is directly controlled, and the direction of the thrust of each propeller is normal to the quadrotor plane. The first and third propellers are assumed to generate a thrust along the direction of $-\vec b_3$ when rotating clockwise; the second and fourth propellers are assumed to generate a thrust along the same direction of $-\vec b_3$ when rotating counterclockwise. Thus, the thrust magnitude is $f=\sum_{i=1}^4 f_i$, and it is positive when the total thrust vector acts  along $-\vec b_3$, and it is negative when the total thrust vector acts along $\vec b_3$. By the definition of the rotation matrix $R\in\SO$, the direction of the $i$-th body-fixed axis $\vec b_i$ is given by $Re_i$ in the inertial frame, where $e_1=[1;0;0],e_2=[0;1;0],e_3=[0;0;1]\in\Re^3$. Therefore, the total thrust vector is given by $-fRe_3\in\Re^3$ in the inertial frame.

We also assume that the torque generated by each propeller is directly proportional to its thrust. Since it is assumed that the first and the third propellers rotate clockwise and the second and the fourth propellers rotate counterclockwise to generate a positive thrust along the direction of $-\vec b_3$, %
the torque generated by the $i$-th propeller about $\vec b_3$ can be written as $\tau_i=(-1)^{i} c_{\tau f} f_i$  for a fixed constant $c_{\tau f}$.    All of these assumptions are fairly common in many quadrotor control systems~\cite{TayMcGITCSTI06,CasLozICSM05}.  %The presented control system can readily be extended to include linear rotor dynamics, as studied in~\cite{BouSiePIICRA05}.

Under these assumptions, the thrust of each propeller $f_1, f_2, f_3, f_4$ is directly converted into $f$ and $M$, or vice versa. In this paper, the thrust magnitude $f\in\Re$ and the moment vector $M\in\Re^3$ are viewed as control inputs. The corresponding equations of motion are given by
\begin{gather}
\dot x  = v,\label{eqn:EL1}\\
m \dot v = mge_3 - f R e_3 + \mathds{W}_{x}(x,v,R,\Omega)\theta_{x},\label{eqn:EL2}\\
\dot R = R\hat\Omega,\label{eqn:EL3}\\
J\dot \Omega + \Omega\times J\Omega = M + \mathds{W}_{R}(x,v,R,\Omega)\theta_{R},\label{eqn:EL4}
\end{gather}
where the \textit{hat map} $\hat\cdot:\Re^3\rightarrow\SO$ is defined by the condition that $\hat x y=x\times y$ for all $x,y\in\Re^3$. 
More explicitly, for a vector $x=[x_1,x_2,x_3]^{T}\in\Re^3$, the matrix $\hat x$ is given by
\begin{align}
    \hat x = \begin{bmatrix} 0 & -x_3 & x_2\\
                                x_3 & 0 & -x_1\\
                                -x_2 & x_1 & 0 \end{bmatrix}\label{eqn:hat}.
\end{align}
This identifies the Lie algebra $\SO$ with $\Re^3$ using the vector cross product in $\Re^3$. %(see Appendix \ref{app:hat}). 
The inverse of the hat map is denoted by the \textit{vee} map, $\vee:\SO\rightarrow\Re^3$. The modeling error and uncertainties in the translational dynamics and the rotational dynamics are given by $\mathds{W}_{x}(x,v,R,\Omega)\theta_{x}$, and $\mathds{W}_{R}(x,v,R,\Omega)\theta_{R}$, respectively. Where $\mathds{W}_{x}(x,v,R,\Omega),\mathds{W}_{R}(x,v,R,\Omega)\in\Re^{3\times P}$ are known functions of the state, and $\theta_{x},\theta_{R}\in\Re^{P\times 1}$ are fixed unknown parameters. It is assumed that the bounds of unknown parameters are given by
\begin{align}
\|\mathds{W}_{x}\| \leq B_{W_{x}},\quad \|\theta_{x}\| \leq B_{\theta},\quad \|\theta_{R}\| \leq B_{\theta},
\end{align}
for $B_{W_{x}},B_{\theta}>0$. Throughout this paper, $\lambda_m (A)$ and $\lambda_{M}(A)$ denote the minimum eigenvalue and the maximum eigenvalue of a square matrix $A$, respectively, and $\lambda_m$ and $\lambda_M$ are shorthand for $\lambda_m=\lambda_m(J)$ and $\lambda_M=\lambda_M(J)$. The two-norm of a matrix $A$ is denoted by $\|A\|$.
%%%%%%%%%%%%%%%%%%%%%%%%%%%%%%%%%%%%%%%%%%%%%%%%%%%%%%%%%%%%%%%%%%%%%%
\section{ATTITUDE CONTROLLED FLIGHT MODE}\label{sec:ACFM}
Since the quadrotor UAV has four inputs, it is possible to achieve asymptotic output tracking for at most four quadrotor UAV outputs.    The quadrotor UAV has three translational and three rotational degrees of freedom; it is not possible to achieve asymptotic output tracking of both attitude and position of the quadrotor UAV. This  motivates us to introduce two flight modes, namely (1) an attitude controlled flight mode, and (2) a position controlled flight mode. While a quadrotor UAV is under-actuated, a complex flight maneuver can be defined by specifying a concatenation of flight modes together with conditions for switching between them. This will be further illustrated by a numerical and experimental examples later. In this section, an attitude controlled flight mode is considered. 
\subsection{Attitude Tracking Errors}
Suppose that an  smooth attitude command $R_d(t)\in\SO$ satisfying the following kinematic equation is given:
\begin{align}
\dot R_d = R_d \hat\Omega_d,
\end{align}
where $\Omega_d(t)$ is the desired angular velocity, which is assumed to be uniformly bounded. We first define errors associated with the attitude dynamics as follows~\cite{BulLew05,LeeITCST13}.
\begin{prop}\label{prop:1}
For a given tracking command $(R_d,\Omega_d)$, and the current attitude and angular velocity $(R,\Omega)$, we define an attitude error function $\Psi:\SO\times\SO\rightarrow\Re$, an attitude error vector $e_R\in\Re^3$, and an angular velocity error vector $e_\Omega\in \Re^3$ as follows~\cite{LeeITCST13}:
\begin{gather}
\Psi (R,R_d) = \frac{1}{2}\tr{I-R_d^TR},\\
e_R =\frac{1}{2} (R_d^TR-R^TR_d)^\vee,\label{eqn:errr}\\
e_\Omega = \Omega - R^T R_d\Omega_d\label{eqn:eomega},
\end{gather}
Then, the following properties hold:
\begin{itemize}
\item[(i)] $\Psi$ is positive-definite about $R=R_d$.
\item[(ii)] The left-trivialized derivative of $\Psi$ is given by
\begin{align}
\T^*_I \mathsf{L}_R\, (\D_R\Psi(R,R_d))= e_R.
\end{align}
\item[(iii)] The critical points of $\Psi$, where $e_R=0$, are $\{R_d\}\cup\{R_d\exp (\pi \hat s),\,s\in\Sph^2 \}$.
\item[(iv)] A lower bound of $\Psi$ is given as follows:
\begin{align}
\frac{1}{2}\|e_R\|^2 \leq \Psi(R,R_d),\label{eqn:PsiLB}
\end{align}
\item[(v)] Let $\psi$ be a positive constant that is strictly less than $2$. If $\Psi(R,R_d)< \psi<2$, then an upper bound of $\Psi$ is given by
\begin{align}
\Psi(R,R_d)\leq \frac{1}{2-\psi} \|e_R\|^2.\label{eqn:PsiUB}
\end{align}
\item[(vi)] The time-derivative of $\Psi$ and $e_R$ satisfies:
\begin{align}
\dot\Psi = e_R\cdot e_\Omega,\quad \|\dot e_R\|\leq\|e_\Omega\|.\label{eqn:Psidot00}
\end{align}
\end{itemize}
\end{prop}
\begin{proof} 
See \cite{LeeITCST13}.
\end{proof}
\subsection{Attitude Tracking Controller}
We now introduce a nonlinear controller for the attitude controlled flight mode:
\begin{align}\label{eqn:aM}
M  =& -k_R e_R -k_\Omega e_\Omega -\mathds{W}_{R}\bar{\theta}_{R}+(R^TR_d\Omega_d)^\wedge J R^T R_d \Omega_d\nonumber \\ 
&+ J R^T R_d\dot\Omega_d,
\end{align}
and adaptive law
\begin{align}\label{eqn:eI}
\dot{\bar{\theta}}_{R}=\gamma_{R}\mathds{W}_{R}^{T}(e_{\Omega}+c_{2}e_{R}),
\end{align}
where $k_R,k_\Omega,c_2,\gamma_{R}$ are positive constants and $\bar{\theta}_{R}\in\Re^{p}$ denotes the estimated value of $\theta_{R}$. The control moment is composed of proportional, derivative, and adaptive terms, augmented with additional terms to cancel out the angular acceleration caused by the desired angular velocity. 
\begin{prop}{(Attitude Controlled Flight Mode)}\label{prop:Att}
Consider the control moment $M$ defined in \refeqn{aM}-\refeqn{eI}. For positive constants $k_R,k_\Omega$, the constants $c_2,B_2$ are chosen such that
\begin{gather}
\|(2J-\trs{J}I)\| \|\Omega_d\| \leq B_2,\label{eqn:B_2}\\
c_2 < \min\bigg\{  \frac{\sqrt{k_R\lambda_m}}{\lambda_M}, \frac{4k_\Omega}{8k_R\lambda_M+(k_\Omega+B_2)^2}\bigg\},\label{eqn:c2}
\end{gather}
the zero equilibrium of tracking errors $(e_R,e_\Omega)=(0,0)$ is stable in the sense of Lyapunov, and $e_R,e_\Omega\rightarrow 0$ as $t\rightarrow\infty$, and furthermore $\tilde{\theta}_{R}$ is uniformly bounded.
\end{prop}
\begin{proof}
See Appendix \ref{sec:pfAtt}.
\end{proof}

While these results are developed for the attitude dynamics of a quadrotor UAV, they can be applied to the attitude dynamics of any rigid body. Nonlinear Adaptive controllers have been developed for attitude stabilization in terms of modified Rodriguez parameters~\cite{SubJAS04} and quaternions~\cite{SubAkeJGCD04}, and for attitude tracking in terms of Euler-angles~\cite{ShoJuaPACC02}. The proposed tracking control system is developed on $\SO$, therefore it avoids singularities of Euler-angles and Rodriguez parameters, as well as unwinding of quaternions. %It also provides almost global asymptotic stability for attitude \textit{tracking} problems with fixed uncertainties.

Asymptotic tracking of the quadrotor attitude does not require specification of the thrust magnitude. As an auxiliary problem, the thrust magnitude can be chosen in many different ways to achieve an additional translational motion objective. For example, it can be used to asymptotically track a quadrotor altitude command~\cite{LeeLeoAJC13}. Since the translational motion of the quadrotor UAV can only be partially controlled; this flight mode is most suitable for short time periods where an attitude maneuver is to be completed. 
\section{POSITION CONTROLLED FLIGHT MODE}\label{sec:PCFM}
We introduce a nonlinear controller for the position controlled flight mode in this section.
\subsection{Position Tracking Errors}
Suppose that an arbitrary smooth position tracking command $x_d(t) \in \Re^3$ is given. The position tracking errors for the position and the velocity are given by:
\begin{align}
e_x  = x - x_d,\quad
e_v  = \dot e_x = v - \dot x_d.
\end{align}
In the position controlled tracking mode, the attitude dynamics is controlled to follow the computed attitude $R_c(t)\in\SO$ and the computed angular velocity $\Omega_c(t)$ defined as
\begin{align}
R_c=[ b_{1_c};\, b_{3_c}\times b_{1_c};\, b_{3_c}],\quad \hat\Omega_c = R_c^T \dot R_c\label{eqn:RdWc},
\end{align}
where $b_{3_c}\in\Sph^2$ is given by
\begin{align}
 b_{3_c} = -\frac{-k_x e_x - k_v e_v -\mathds{W}_{x}\bar{\theta}_{x} - mg e_3 +m\ddot x_d}{\norm{-k_x e_x - k_v e_v -\mathds{W}_{x}\bar{\theta}_{x}- mg e_3 + m\ddot x_d}},\label{eqn:Rd3}
\end{align}
for positive constants $k_x,k_v$ and $\bar{\theta}_{x}\in\Re^{p}$ denotes the estimate value of $\theta_{x}$. The unit vector $b_{1_c}\in\Sph^2$ is selected to be orthogonal to $b_{3_c}$, thereby guaranteeing that $R_c\in\SO$. It can be chosen to specify the desired heading direction, and the detailed procedure to select $b_{1c}$ is described later at Section \ref{sec:b1c}. 
%Since $e_v=\dot e_x$, this can be rewritten as $e_i(t) = e_x(t) - e_x(0) + c_1 \int_0^t e_x(\tau)d\tau$. 
Following the prior definition of the attitude error and the angular velocity error given at \refeqn{errr} and \refeqn{eomega}, and assume that the commanded acceleration is uniformly bounded:
\begin{align}
\|-mge_3+m\ddot x_d\| < B_1\label{eqn:B}
\end{align}
for a given positive constant $B_1$. These imply that the given desired position command is distinctive from free-fall, where no control input is required.%It is also assumed that an upper bound of the infinite norm of the uncertainty is known:
%\begin{align}
%\|\bar{\theta}_{x}\|\leq \sqrt{3}\sigma
%\end{align}
\subsection{Position Tracking Controller}
The nonlinear controller for the position controlled flight mode, described by control expressions for the  thrust magnitude and the  moment vector, are:
\begin{align}
f  =& ( k_x e_x + k_v e_v +\mathds{W}_{x}\bar{\theta}_{x}+ mg e_3-m\ddot x_d)\cdot Re_3,\label{eqn:f}\\
M  =& -k_R e_R -k_\Omega e_\Omega -\mathds{W}_{R}\bar{\theta}_{R}+(R^TR_c\Omega_c)^\wedge J R^T R_c \Omega_c \nonumber\\
&+ J R^T R_c\dot\Omega_c.\label{eqn:M}
\end{align}
Similar with \refeqn{eI}, an adaptive control law for the position tracking controller is defined as
\begin{align}\label{eqn:adaptivelawx}
\dot{\bar{\theta}}_{x}&=\begin{cases}
\gamma_{x}\mathds{W}_{x}^{T}(e_{v}+c_{1}e_{x}) & \hspace{-0.4cm}\mbox{if } \|\bar{\theta}_{x}\|<B_{\theta}\\
& \hspace{-0.4cm}\mbox{or } \|\bar{\theta}_{x}\|=B_{\theta}\\
& \hspace{-0.4cm}\mbox{and } \bar{\theta}_{x}^{T}\mathds{W}_{x}^{T}(e_{v}+c_{1}e_{x})\leq 0\\
&\\
\gamma_{x}(I-\frac{\bar{\theta}_{x}\bar{\theta}_{x}^{T}}{\bar{\theta}_{x}^{T}\bar{\theta}_{x}})\mathds{W}_{x}^{T}(e_{v}+c_{1}e_{x}) & \mbox{Otherwise}
\end{cases},
\end{align}
for a positive constants $c_1$ and $\gamma_{x}$. The above expression correspond adaptive controls with projection~\cite{Ioannou96} and it is included to restrict the effects of attitude tracking errors on the translational dynamics.

%considers two cases: if $\bar{\theta}_{x}$ is inside or on the layer of the boundary, the $\dot{\bar{\theta}}_{x}$ is well-defined in the first case of \refeqn{adaptivelawx}, otherwise $\frac{\bar{\theta}_{x}\bar{\theta}_{x}^T}{\bar{\theta}_{x}^T\bar{\theta}_{x}}$  is the projection into the plane normal to $\bar{\theta}_{x}$ when $\dot{\bar{\theta}}_{x}$ goes beyond the defined boundary, it projects the result into the layer of the boundary.% as shown in Figure \ref{fig:boundary}.
%\begin{figure}
%\centerline{		
%\includegraphics[width=0.5\columnwidth]{boundary.pdf}
%}
%\caption{Adaptive law boundary}\label{fig:boundary}
%\end{figure}
The nonlinear controller given by equations \refeqn{f}, \refeqn{M} can be given a backstepping interpretation. The computed attitude $R_c$ given in equation \refeqn{RdWc} is selected so that the thrust axis $-b_3$ of the quadrotor UAV tracks the computed direction given by $-b_{3_c}$ in \refeqn{Rd3}, which is a direction of the thrust vector that achieves position tracking. The moment expression \refeqn{M} causes the attitude of the quadrotor UAV to asymptotically track $R_c$ and the thrust magnitude expression \refeqn{f} achieves asymptotic position tracking. 
%The corresponding closed loop control system is described by equations \refeqn{EL1}-\refeqn{EL4}, using the controller expressions \refeqn{f}-\refeqn{M}. We now state the result that the zero equilibrium of tracking errors $(e_x, e_v, e_R, e_\Omega)$ is exponentially stable.
\begin{prop}{(Position Controlled Flight Mode)}\label{prop:Pos}
Suppose that the initial conditions satisfy
\begin{align}\label{eqn:Psi0}
\Psi(R(0),R_c(0)) < \psi_1 < 1,
\end{align}
for positive constant $\psi_1$. Consider the control inputs $f,M$ defined in \refeqn{f}-\refeqn{M}.
For positive constants $k_x,k_v$, we choose positive constants $c_1,c_2,k_R,k_\Omega$ such that
\begin{gather}
%k_i\sigma > \delta_x,\label{eqn:kisigma}\\
c_1 < \min\braces{\frac{4k_xk_v(1-\alpha)^2}{k_v^2(1+\alpha)^2+4m k_x(1-\alpha)},\; \sqrt{\frac{k_x}{m}} },\label{eqn:c1b}\\
\lambda_{m}(W_2) > \frac{\|W_{12}\|^2}{4\lambda_{m}(W_1)},\label{eqn:kRkWb}
\end{gather}
and \refeqn{c2} is satisfied, where $\alpha=\sqrt{\psi_1(2-\psi_1)}$, and the matrices $W_1,W_{12},W_2\in\Re^{2\times 2}$ are given by
\begin{align}
W_1 &= \begin{bmatrix} {c_1k_x}(1-\alpha) & -\frac{c_1k_v}{2}(1+\alpha)\\
-\frac{c_1k_v}{2}(1+\alpha) & k_v(1-\alpha)-mc_1\end{bmatrix},\label{eqn:W1}\\
W_{12}&=\begin{bmatrix}
{c_1}(B_{W_{x}} B_{\theta}+B_1) & 0 \\ B_{W_{x}} B_{\theta}+B_1+k_xe_{x_{\max}} & 0\end{bmatrix},\label{eqn:W12}\\
W_2 & = \begin{bmatrix} c_2k_R & -\frac{c_2}{2}(k_\Omega+B_2) \\ 
-\frac{c_2}{2}(k_\Omega+B_2) & k_\Omega-2c_2\lambda_M \end{bmatrix}.\label{eqn:W2}
\end{align}
This implies that the zero equilibrium of the tracking error is stable in the sense of Lyapunov and the tracking error variables asymptotically converge to zero. Also, the estimation errors are uniformly bounded.
%Then, the zero equilibrium of the tracking errors is exponentially stable with respect to $e_x,e_v,e_R,e_\Omega$.
\end{prop}
\begin{proof}
See Appendix \ref{sec:pfPos}.
\end{proof}
This proposition shows that the proposed control system is robust to unstructured uncertainties in the dynamics of a quadrotor UAV, and in the presence of uncertainties, the tracking error variables still converge to zero. 

Proposition \ref{prop:Pos} requires that the initial attitude error is less than $90^\circ$ in \refeqn{Psi0}. Suppose that this is not satisfied, i.e. $1\leq\Psi(R(0),R_c(0))<2$. We can still apply Proposition \ref{prop:Att}, which states that the attitude error is asymptotically decreases to zero for almost all cases, and it satisfies \refeqn{Psi0} in a finite time. Therefore, by combining the results of Proposition \ref{prop:Att} and \ref{prop:Pos}, we can show attractiveness of the tracking errors when $\Psi(R(0),R_c(0))<2$.

\begin{prop}{(Position Controlled Flight Mode with a Larger Initial Attitude Error)}\label{prop:Pos2}
Suppose that the initial conditions satisfy
\begin{gather}
1\leq \Psi(R(0),R_c(0)) < 2\label{eqn:eRb3},\quad \|e_x(0)\| < e_{x_{\max}},
\end{gather}
for a constant $e_{x_{\max}}$. Consider the control inputs $f,M$ defined in \refeqn{f}-\refeqn{M}, where the control parameters satisfy \refeqn{Psi0}-\refeqn{kRkWb} for a positive constant $\psi_1<1$. Then the zero equilibrium of the tracking errors is attractive, i.e., $e_x,e_v,e_R,e_\Omega\rightarrow 0$ as $t\rightarrow\infty$. 
\end{prop}
\begin{proof}
See Appendix \ref{sec:pfPos2}.
%See \cite{LeeLeo_4457}.
\end{proof}

Linear or nonlinear PID and adaptive controllers have been widely used for a quadrotor UAV. But, they have been applied in an \textit{ad-hoc} manner without stability analysis. This paper provides a new form of nonlinear adaptive controller on $\SE$ that guarantees almost global attractiveness in the presence of uncertainties. The nonlinear robust tracking control system in~\cite{LeeLeoPACC12,LeeLeoAJC13} provides ultimate boundedness of tracking errors, and the control input may be prone to chattering if the required ultimate bound is smaller. Compared with~\cite{LeeLeoAJC13}, the control system in this paper guarantees stronger asymptotic stability, and there is no concern for chattering. 

\subsection{Direction of the First Body-Fixed Axis}\label{sec:b1c}
As described above, the construction of the orthogonal matrix $R_c$ involves having its third column $b_{3_c}$ specified by \refeqn{Rd3}, and its first column $b_{1_c}$ is arbitrarily chosen to be orthogonal to the third column, which corresponds to a one-dimensional degree of choice. 

By choosing $b_{1_c}$ properly, we constrain the asymptotic direction of the first body-fixed axis. Here, we propose to specify the \textit{projection} of the first body-fixed axis onto the plane normal to $b_{3_c}$. In particular, we choose a desired direction $b_{1_d}\in\Sph^2$, that is not parallel to $b_{3_c}$, and $b_{1_c}$ is selected as $b_{1_c}=\mathrm{Proj}[b_{1_d}]$, where $\mathrm{Proj}[\cdot]$ denotes the normalized projection onto the plane perpendicular to $b_{3_c}$. In this case, the first body-fixed axis does not converge to $b_{1_d}$, but it converges to the projection of $b_{1_d}$, i.e. $b_1\rightarrow b_{1_c}=\mathrm{Proj}[b_{1_d}]$ as $t\rightarrow\infty$. This can be used to specify the heading direction of a quadrotor UAV in the horizontal plane~\cite{LeeLeoAJC13}.
\section{NUMERICAL EXAMPLE}\label{sec:NE}
%%%
%Numerical results are presented to demonstrate the prior approach for performing complex flight maneuvers. 
 An aggressive flipping maneuver is considered for a numerical simulation to validate the proposed controller and to show versatile adaptability of the controller due to its large region of attraction. The quadrotor parameters are chosen as
\begin{align*}
J =
\begin{bmatrix}
    5.5711 & 0.0618 & -0.0251\\
    0.06177 & 5.5757 & 0.0101\\
    -0.02502 & 0.01007 & 1.05053
\end{bmatrix}\times 10^{-2}\quad\mathrm{kgm}^2 ,\\ m = 0.755\mathrm{kg},\;\;
d = 0.169\mathrm{m},\;\; c_{\tau f} = 0.1056.
 \end{align*}
Also, controller parameters are selected as follows: $k_x=6.0$, $k_v=3.0$, $k_R=0.7$, $k_\Omega = 0.12$, $c_1=0.1$, $c_2=0.1$.

In this simulation, initial state of the quadrotor UAV is at a hovering condition: $x(0)=v(0)=\Omega(0)=0_{3\times 1}$, and $R(0)=I_{3\times 3}$. The desired trajectory is a flipping maneuver where the quadrotor rotates about rotation axis $e_r = [\frac{\sqrt{2}}{2}, \frac{\sqrt{2}}{2}, 0]$ by $360^\circ$. This is a complex maneuver combining a nontrivial pitching maneuver with a yawing motion. It is achieved by concatenating the following two control modes:
\begin{itemize}
\item[(i)] Attitude tracking to rotate the quadrotor ($t\leq 0.375$)
\begin{align*}
&R_d(t)= I+\sin(4 \pi t)\hat{e}_r+(1-\cos(4 \pi t))(e_r e_r^T-I),\\
&\Omega_d= 4\pi\cdot e_r.
\end{align*}
\item[(ii)] Trajectory tracking to make it hover after completing the preceding rotation ($0.375<t \leq 2$)
\begin{align*}
x_d(t)=[0,0,0]^T,\quad b_{1_d} = [1,0,0]^T.
\end{align*}
\end{itemize}

We considered two cases for this numerical simulation to compare the effect of the proposed adaptive term in the presence of disturbances. In this numerical simulation we considered a special case of $W_{x}=I_{3\times 3}$ and $W_{R}=I_{3\times 3}$. 
%In this case, the adaptive terms become the following integral terms in both translation and rotational controller force and moments
%\begin{align}
%e_{I}&=\int_{0}^{t}{e_{\Omega}(\tau)+c_{2}e_{R}(\tau)}d\tau,\\
%e_{i}&=\int_{0}^{t}{e_{v}(\tau)+c_{2}e_{x}(\tau)}d\tau.
%\end{align}
Two cases are as follows: (i) with adaptive term and (ii) without the adaptive term, where constant disturbances are defined as
\begin{align*}
\theta_R=[0.03,-0.06,0.09]^T,\quad \theta_x=[0.25,0.125,0.2]^T
\end{align*}

The switching time is determined as follows. The desired angular velocity of the first attitude tracking mode is $4\pi$, which means it requires $0.5\;\mathrm{sec}$ for one revolution but the control mode is switched to trajectory tracking mode at $t=0.375\;\mathrm{sec}$ to compensate rotational inertia.

It is also assumed that the maximum thrust at each rotor is given by $f_{i_{\max}}=3.2\,\mathrm{N}$, and any thrust command above the maximum thrust is saturated to represent the actual motor limitation in the numerical simulation.

The attitude error functions are depicted at Figures \ref{fig:sim_error_WO}, \ref{fig:sim_error_W}, and they show that errors jump at $0.375\;\mathrm{sec}$ when the controller is switched, but it converges exponentially to zero according to Proposition 3. Figures \ref{fig:sim_force_WO}, \ref{fig:sim_force_W} illustrate saturated thrust force. 

Comparison between these two cases, as illustrated in Figure \ref{fig:sim_WO} and \ref{fig:sim_W}, shows that the adaptive term eliminates the steady state error, reduces the attitude error variables, and reduces the drop in altitude during the maneuver significantly.

\begin{figure}
\centerline{
	\subfigure[Attitude error function $\Psi$]{
		\includegraphics[width=0.33\columnwidth]{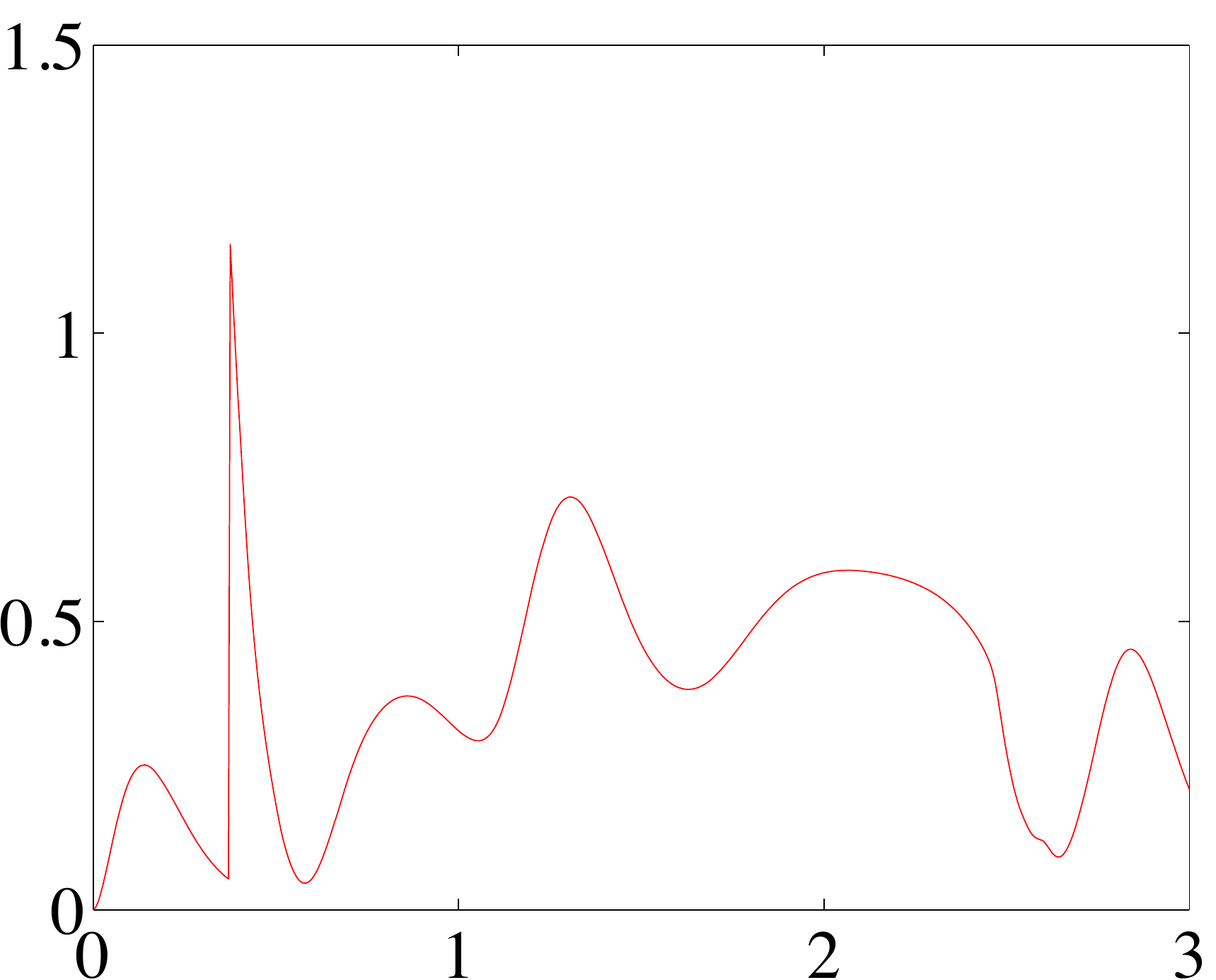}\label{fig:sim_error_WO}}
	\subfigure[Thrust at each rotor $f_i$ ($\mathrm{N}$)]{
		\includegraphics[width=0.34\columnwidth]{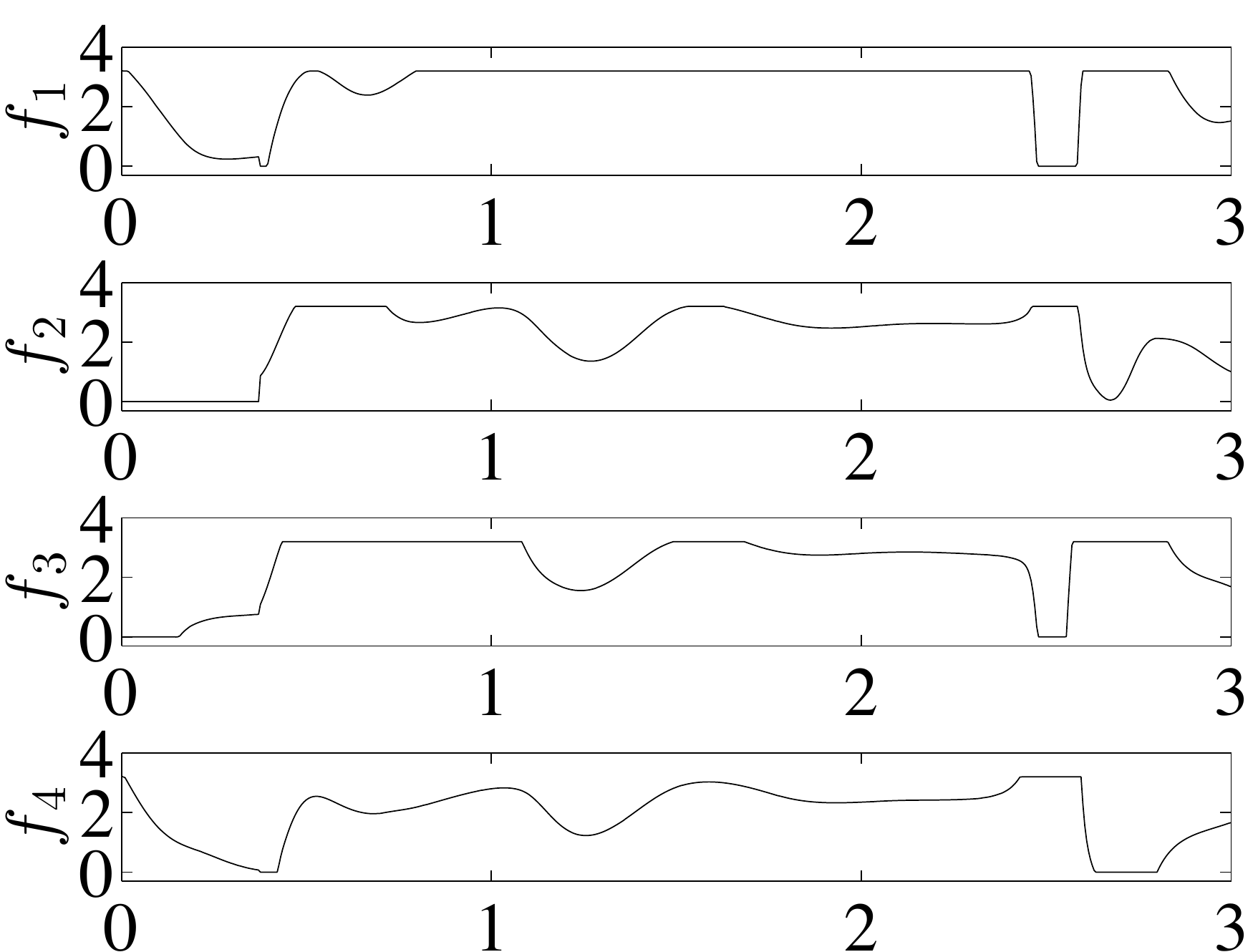}\label{fig:sim_force_WO}}
		\subfigure[Attitude error $e_{R}$ ($\mathrm{rad}$)]{
		\includegraphics[width=0.33\columnwidth]{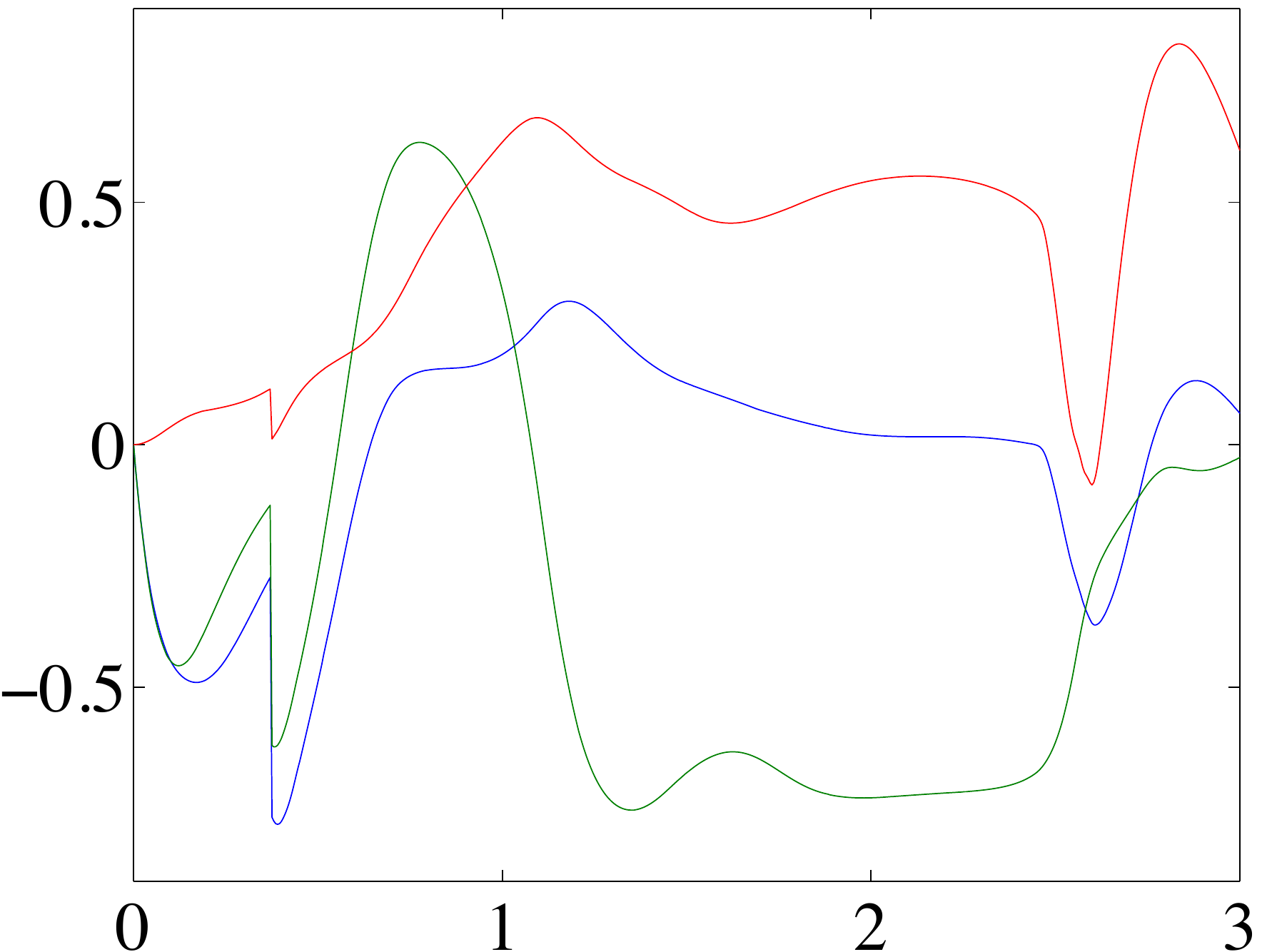}\label{fig:sim_eR_WO}}
}
\centerline{
	\subfigure[Angular Velocity error $e_{\Omega}$($\mathrm{rad}/\mathrm{sec}$)]{
		\includegraphics[width=0.33\columnwidth]{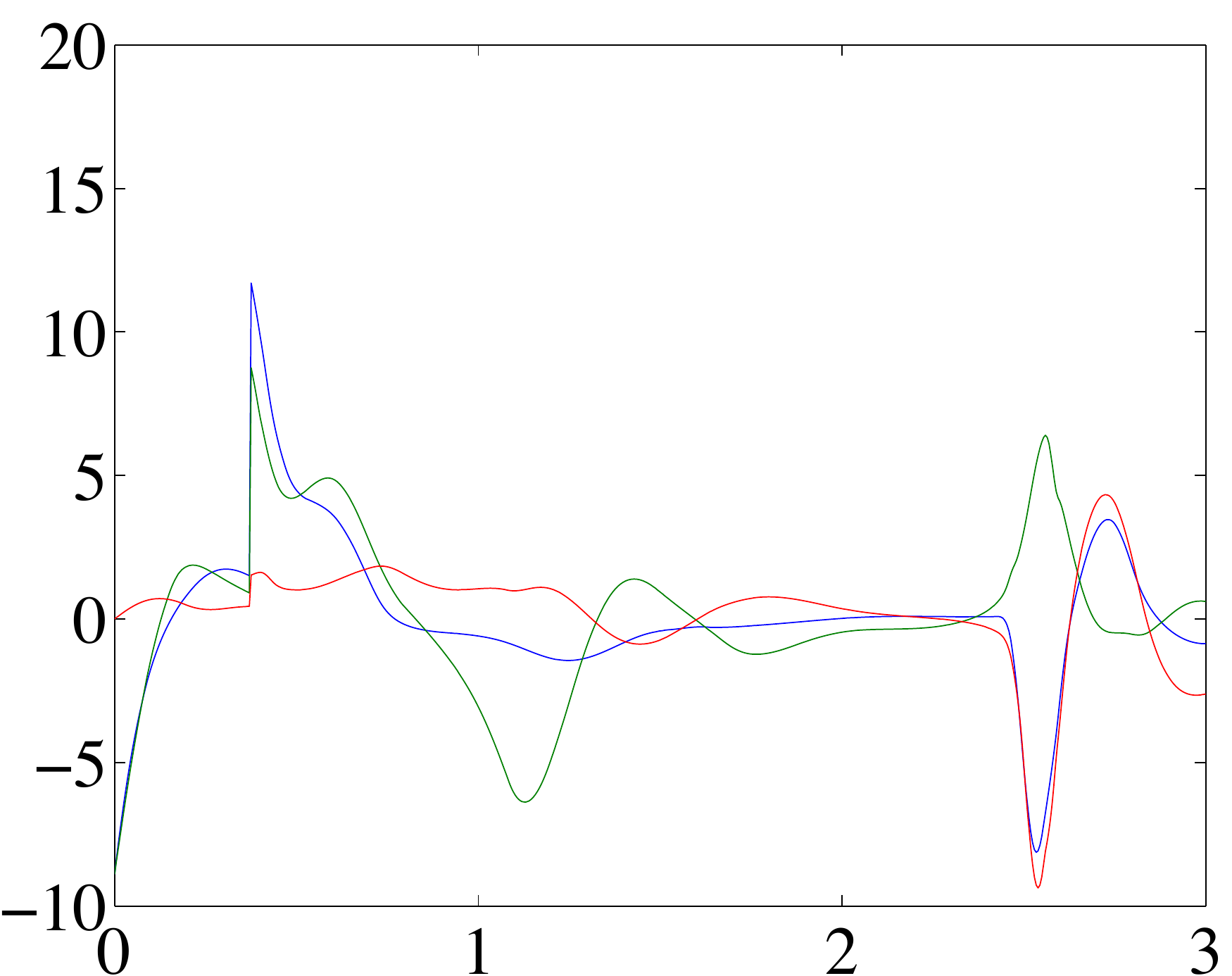}\label{fig:sim_eW_WO}}
		\subfigure[Angular Velocity $\Omega$,$\Omega_{d}$ ($\mathrm{rad/s}$)]{
		\includegraphics[width=0.33\columnwidth]{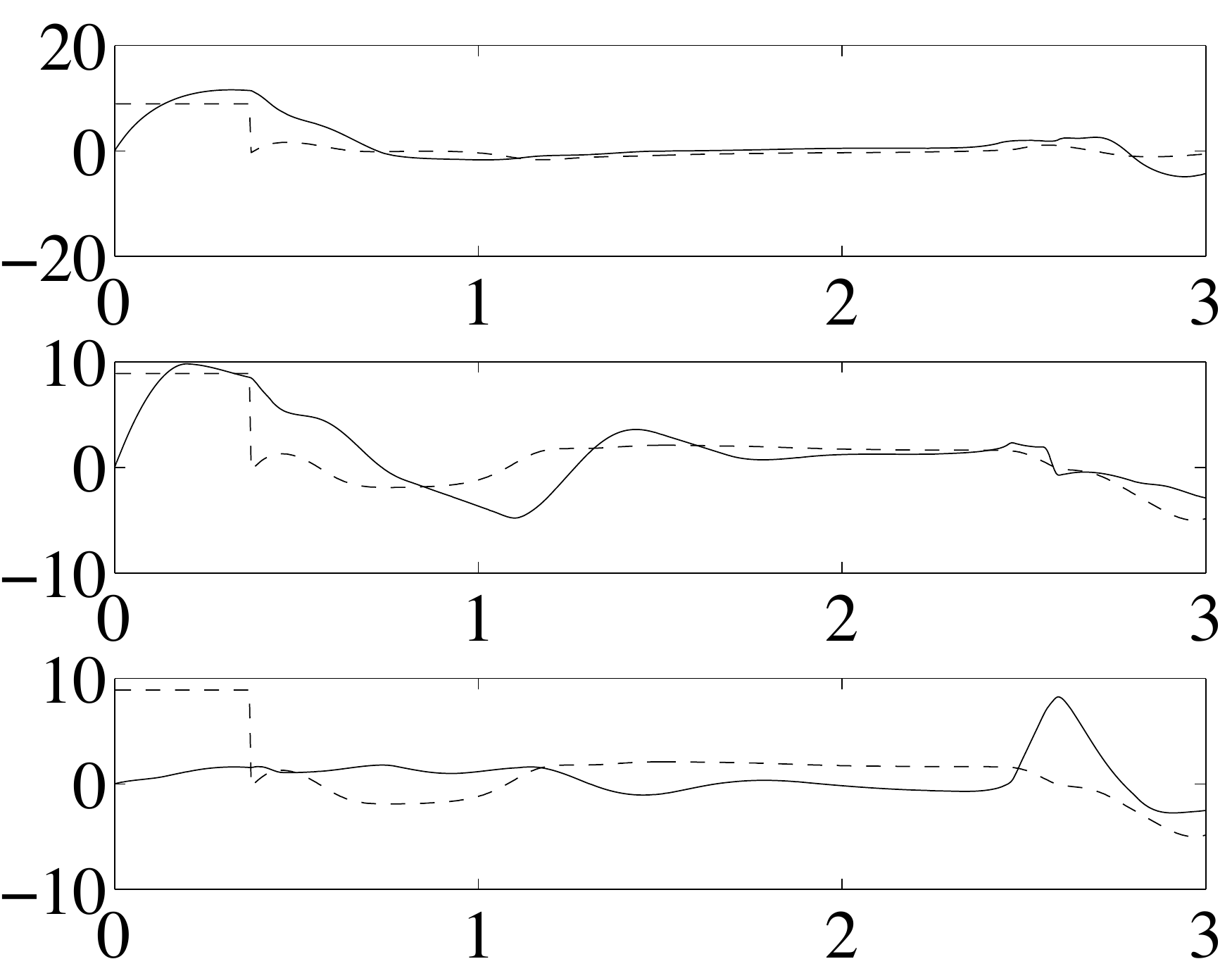}\label{fig:sim_ang_WO}}
	\subfigure[Position $x$,$x_{d}$($\mathrm{m}$)]{
		\includegraphics[width=0.33\columnwidth]{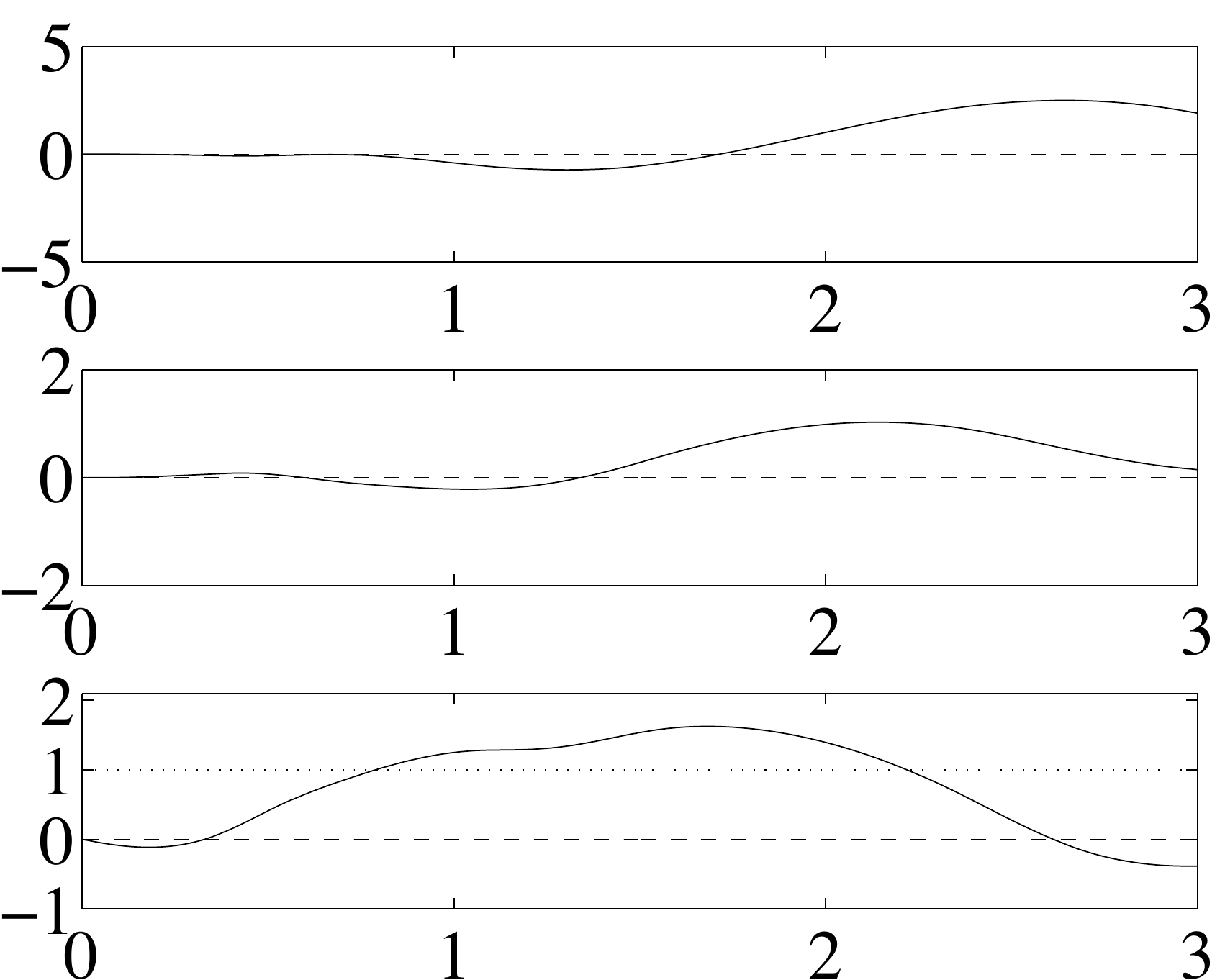}\label{fig:sim_x_WO}}
}
\centerline{
}
\caption{Flipping without adaptive term (dotted:desired, solid:actual)}\label{fig:sim_WO}
\end{figure}

\begin{figure}
\centerline{
	\subfigure[Attitude error function $\Psi$]{
		\includegraphics[width=0.33\columnwidth]{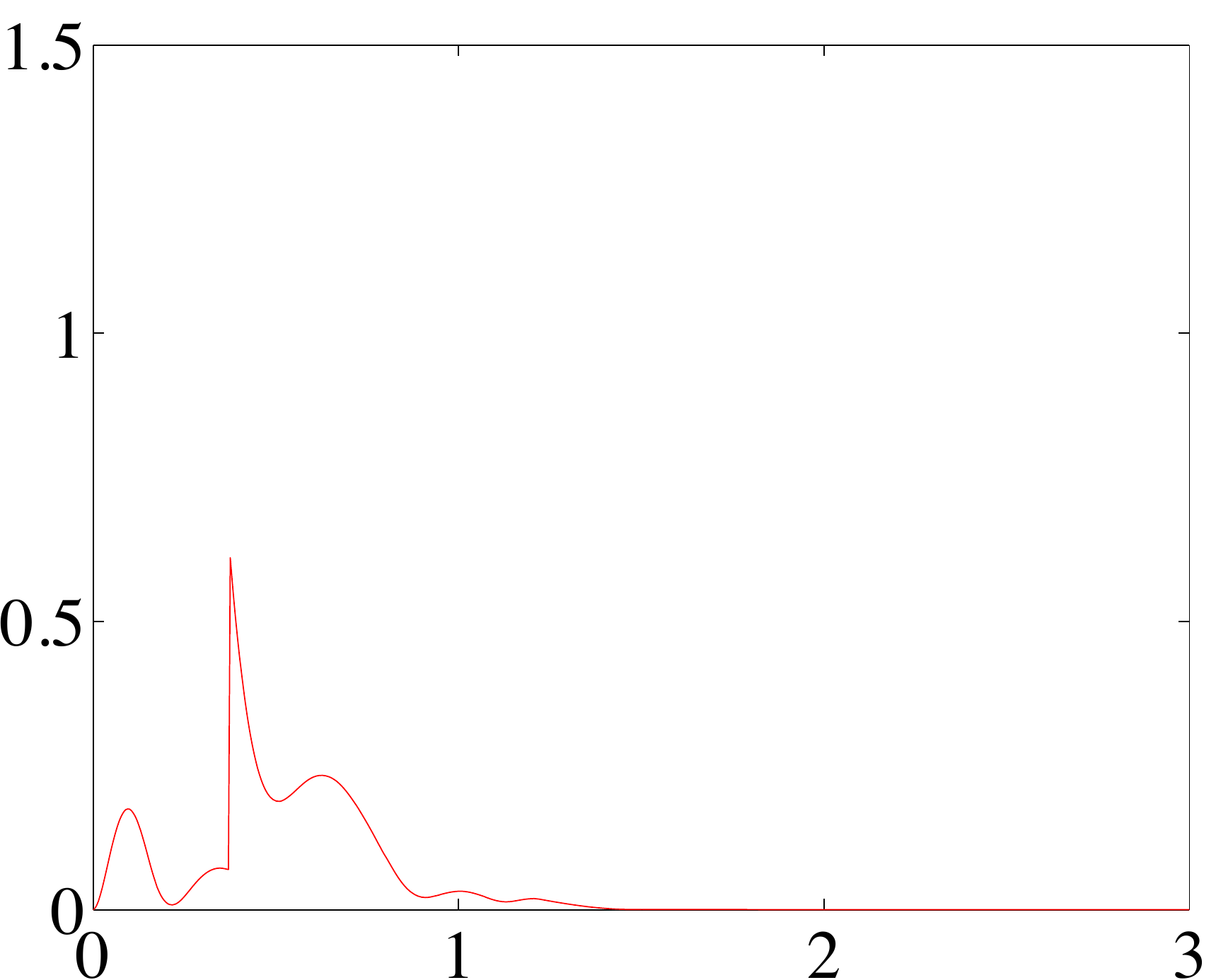}\label{fig:sim_error_W}}
	\subfigure[Thrust at each rotor $f_i$ ($\mathrm{N}$)]{
		\includegraphics[width=0.34\columnwidth]{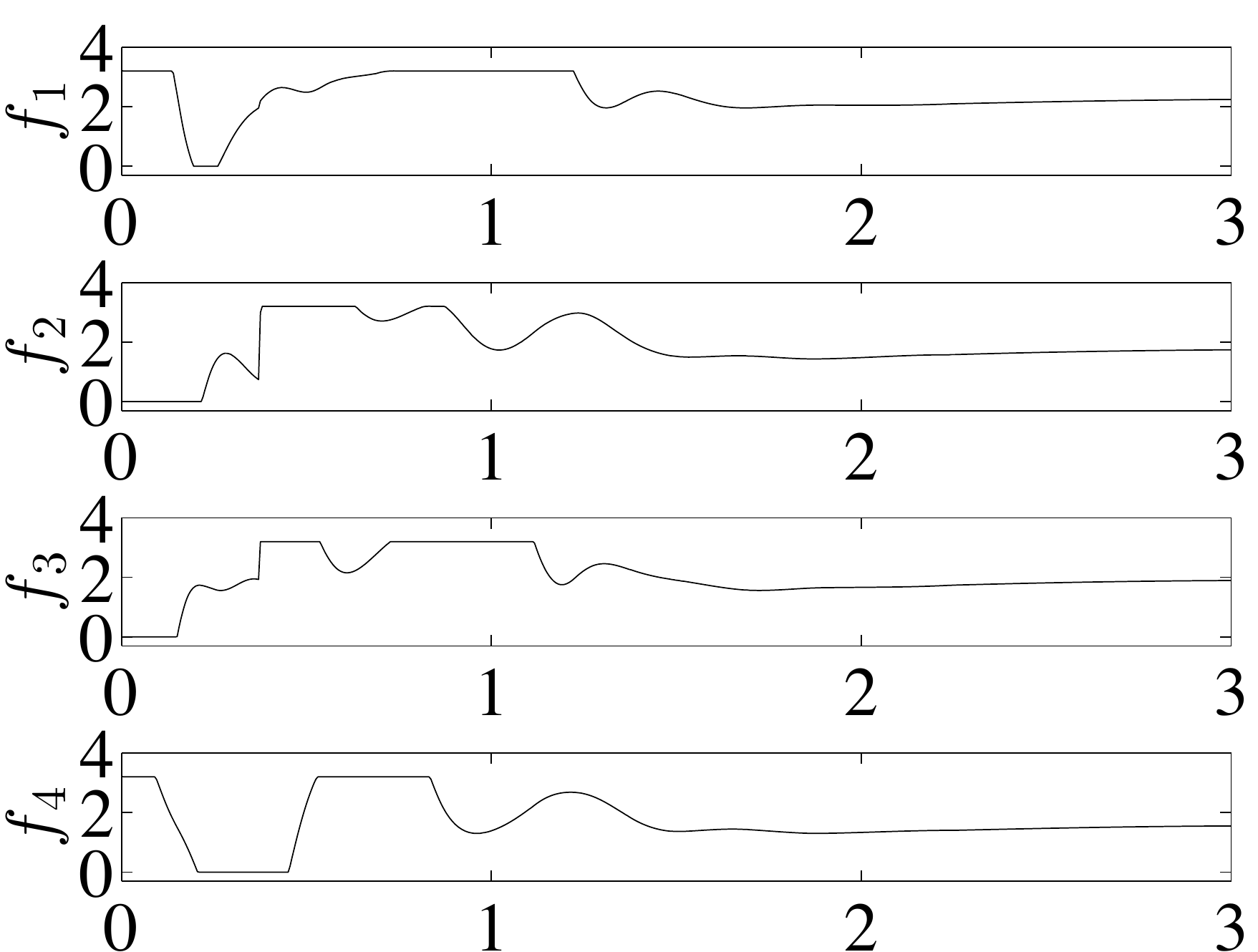}\label{fig:sim_force_W}}
		\subfigure[Attitude error $e_{R}$ ($\mathrm{rad}$)]{
		\includegraphics[width=0.33\columnwidth]{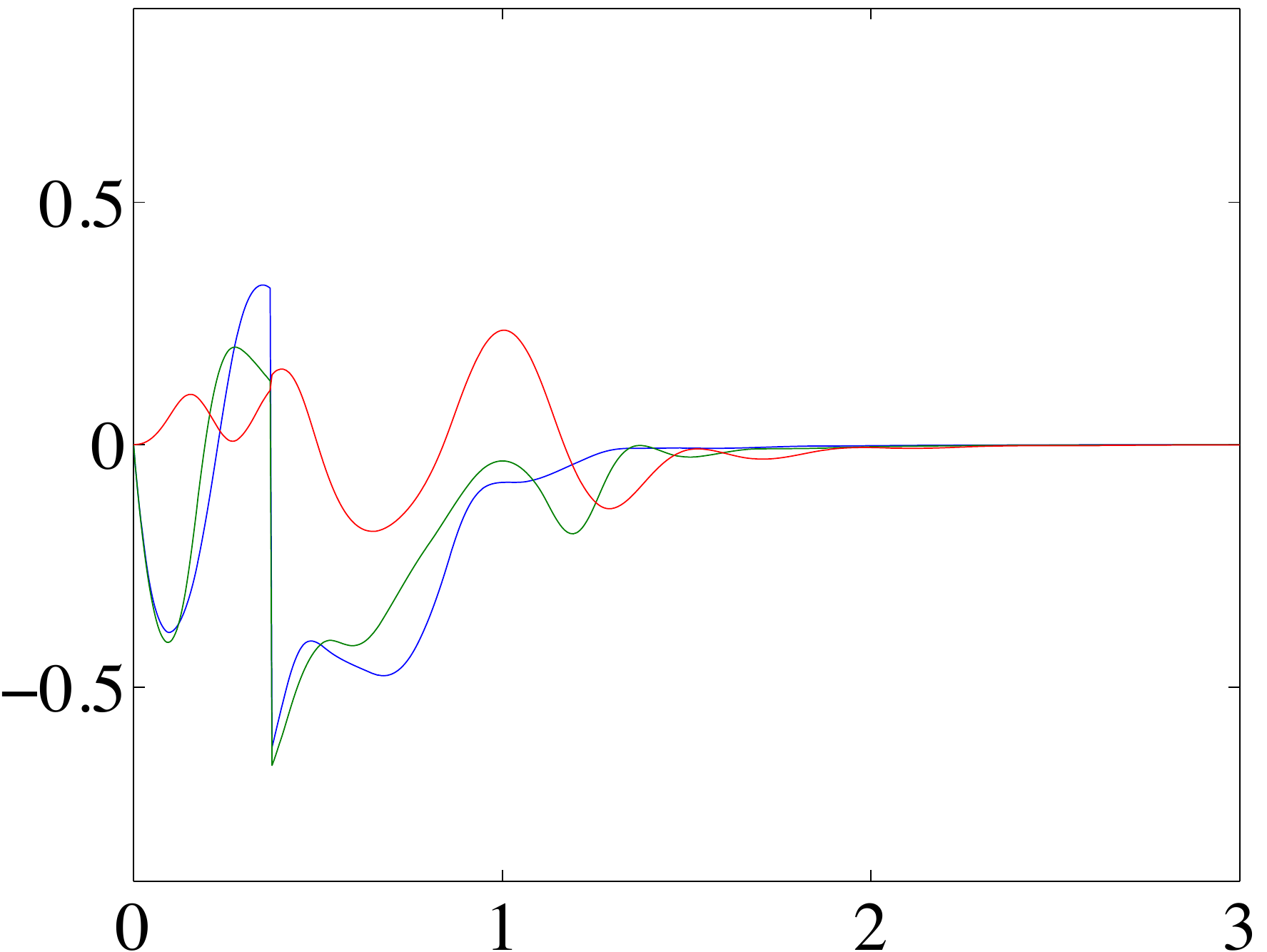}\label{fig:sim_eR_W}}
}
\centerline{
	\subfigure[Angular Velocity error $e_{\Omega}$($\mathrm{rad}/\mathrm{sec}$)]{
		\includegraphics[width=0.33\columnwidth]{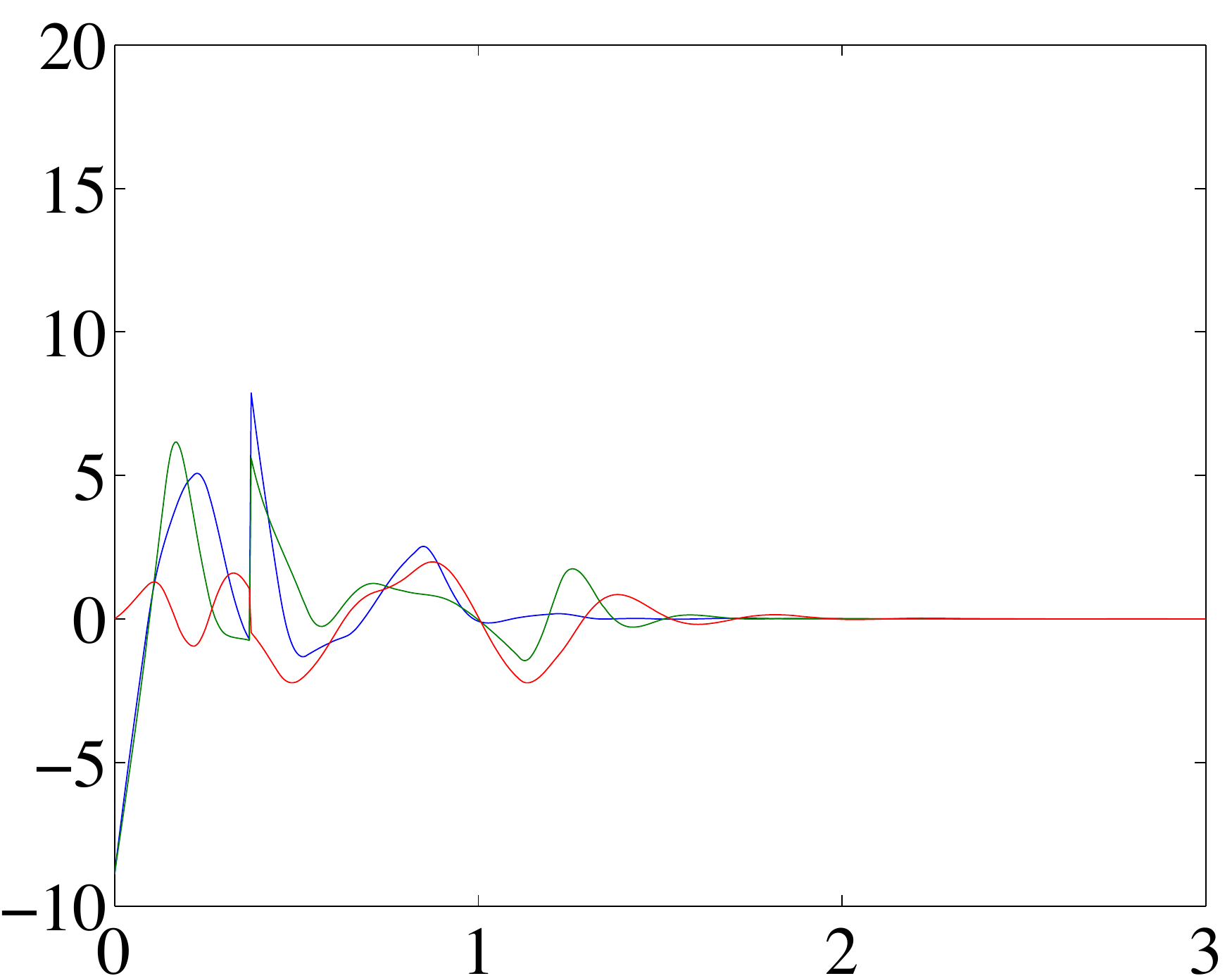}\label{fig:sim_eW_W}}
		\subfigure[Angular Velocity $\Omega$,$\Omega_{d}$ ($\mathrm{rad/s}$)]{
		\includegraphics[width=0.33\columnwidth]{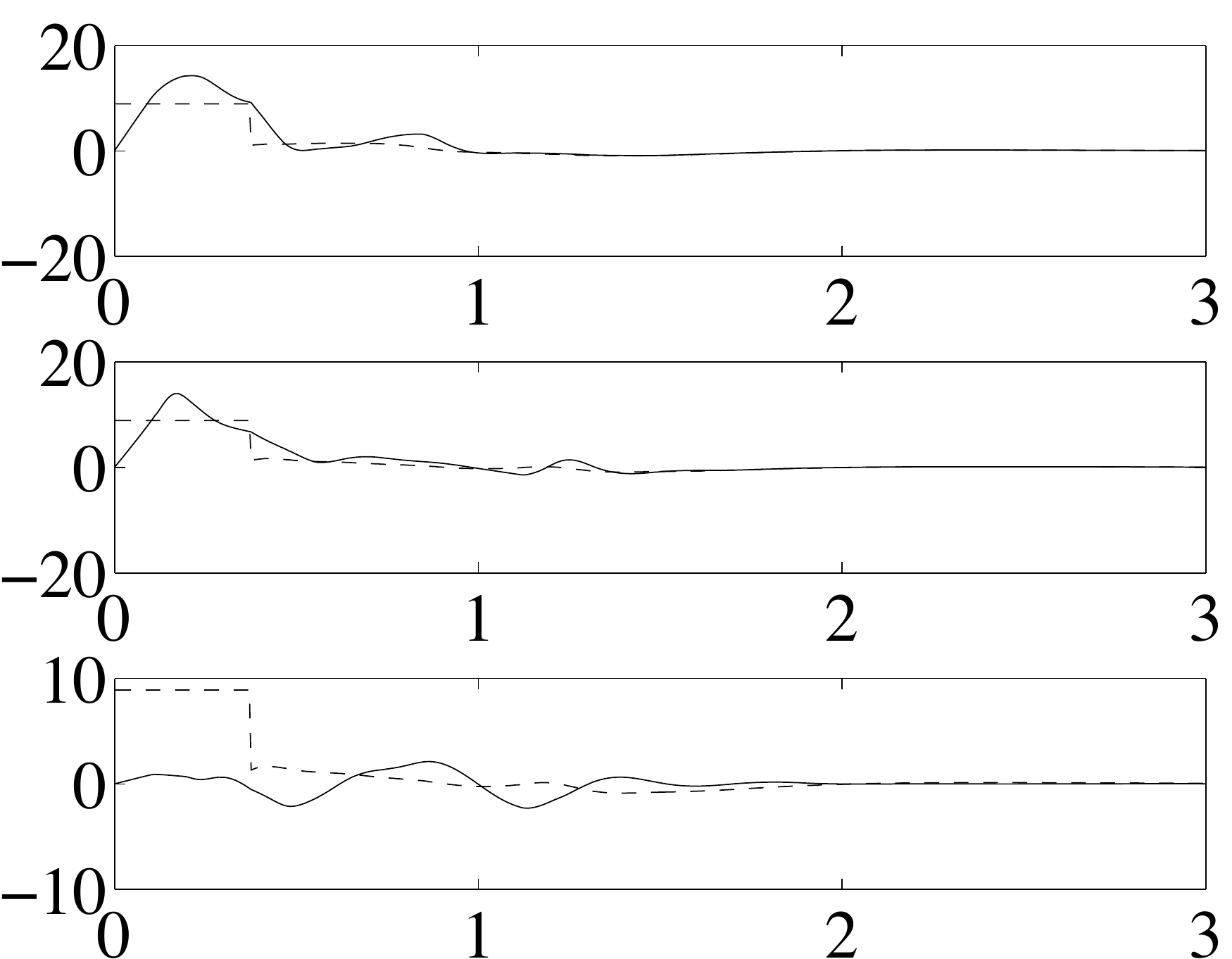}\label{fig:sim_ang_W}}
	\subfigure[Position $x$,$x_{d}$($\mathrm{m}$)]{
		\includegraphics[width=0.33\columnwidth]{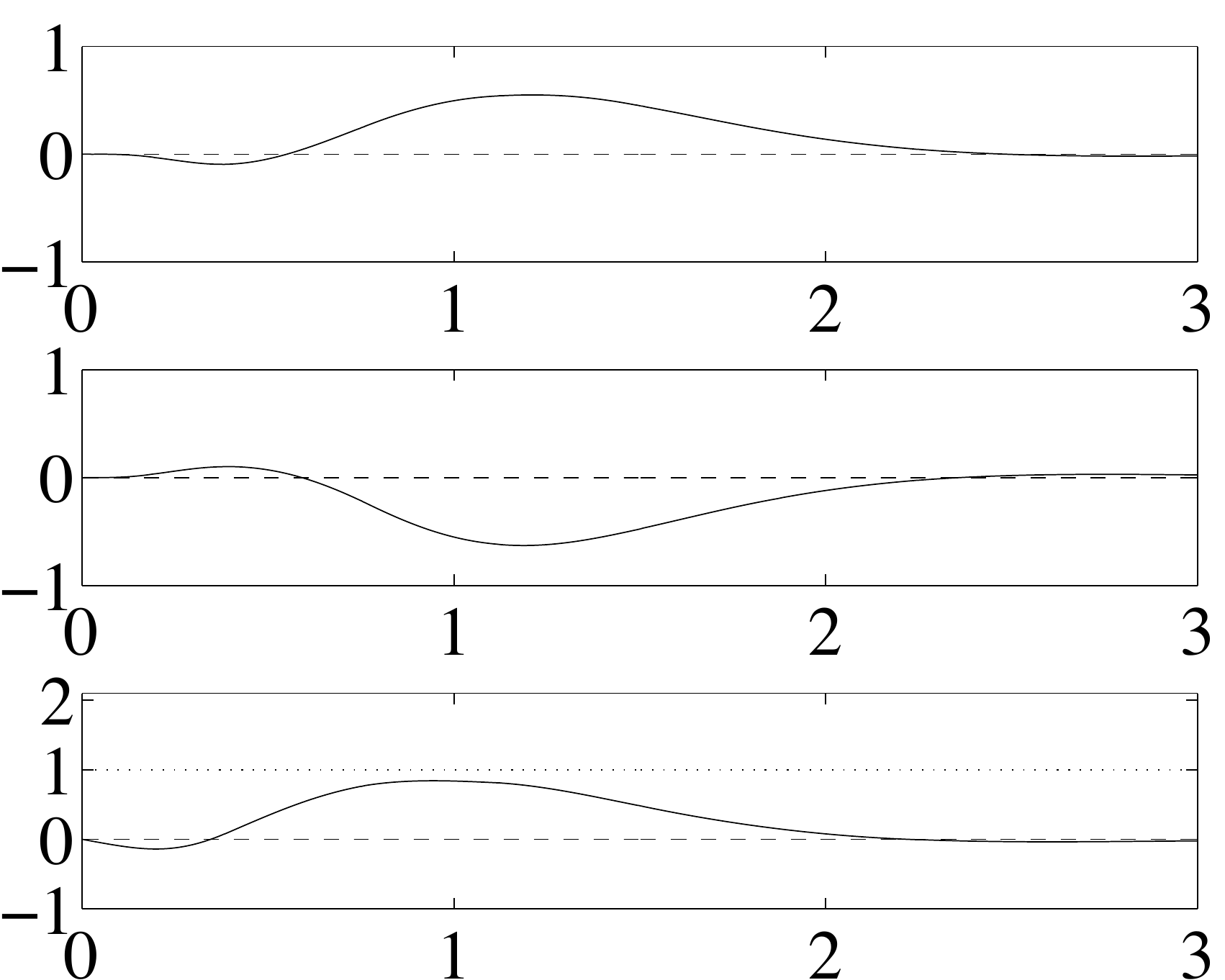}\label{fig:sim_x_W}}
}
\centerline{
}
\caption{Flipping with adaptive term (dotted:desired, solid:actual)}\label{fig:sim_W}
\end{figure}

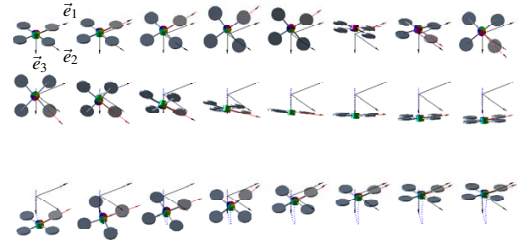
\begin{figure}
\setlength{\unitlength}{0.08\columnwidth}\scriptsize
\centerline{
\begin{picture}(9.9,4.5)(0,0)
	\put(0,0){\includegraphics[width=0.8\columnwidth]{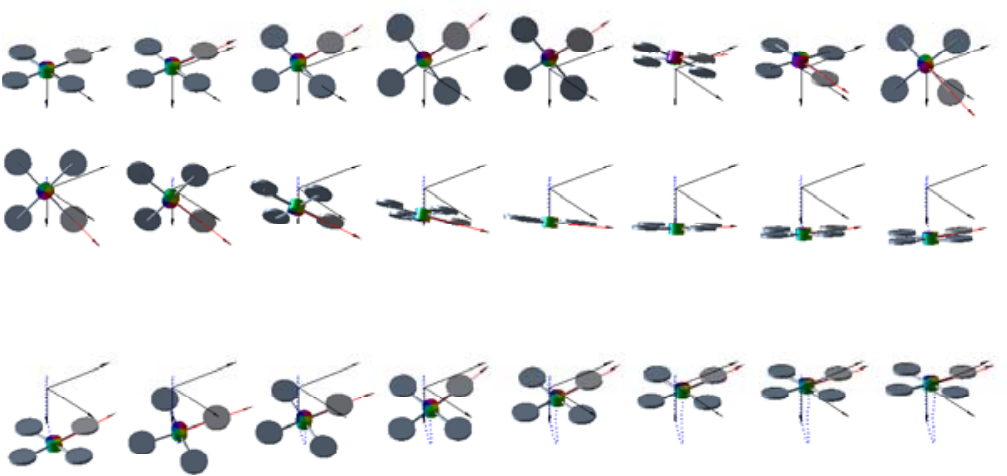}}
	\put(1.00,4.4){$\vec e_1$}
	\put(1.0,3.5){$\vec e_2$}
	\put(0.4,3.35){$\vec e_3$}
\end{picture}}
\caption{Snapshots of a flipping maneuver: the red line denotes the rotation axis $e_r=[\frac{1}{\sqrt{2}},\; \frac{1}{\sqrt{2}},\;0]$. The quadrotor UAV rotates about the $e_r$ axis by $360^\circ$. The trajectory of its mass center is denoted by blue, dotted lines.}\label{fig:Int3D}
\end{figure}

\section{EXPERIMENTAL RESULTS}\label{sec:ER}

In this section, an experimental setup is described and the proposed geometric nonlinear controller is validated with experiments.

\subsection{Hardware Description}\label{HD}

The quadrotor UAV developed at the flight dynamics and control laboratory at the George Washington University is shown in figure~\ref{fig:Quad}, and its parameters are the same as described as the pervious section. The angular velocity is measured from inertial measurement unit (IMU) and the attitude is obtained from IMU data. Position of the UAV is measured from motion capture system (Vicon) and the velocity is estimated from the measurement. Ground computing system receives the Vicon data and send it to the UAV via XBee. The Gumstix is adopted as micro computing unit on the UAV. It has three main threads, namely Vicon thread, IMU thread, and control thread. The Vicon thread receives the Vicon measurement and estimates linear velocity of the quadrotor. In IMU thread, it receives the IMU measurement and estimates the attitude. The last thread handles the control outputs at each time step. Also, control outputs are calculated at 120Hz which is fast enough to run any kind of aggressive maneuvers. Information flow of the system is illustrated in Figure \ref{fig:information_flow}.

\begin{figure}
\centerline{
	\subfigure[Hardware configuration]{
\setlength{\unitlength}{0.1\columnwidth}\scriptsize
\begin{picture}(7,4)(0,0)
\put(0,0){\includegraphics[width=0.7\columnwidth]{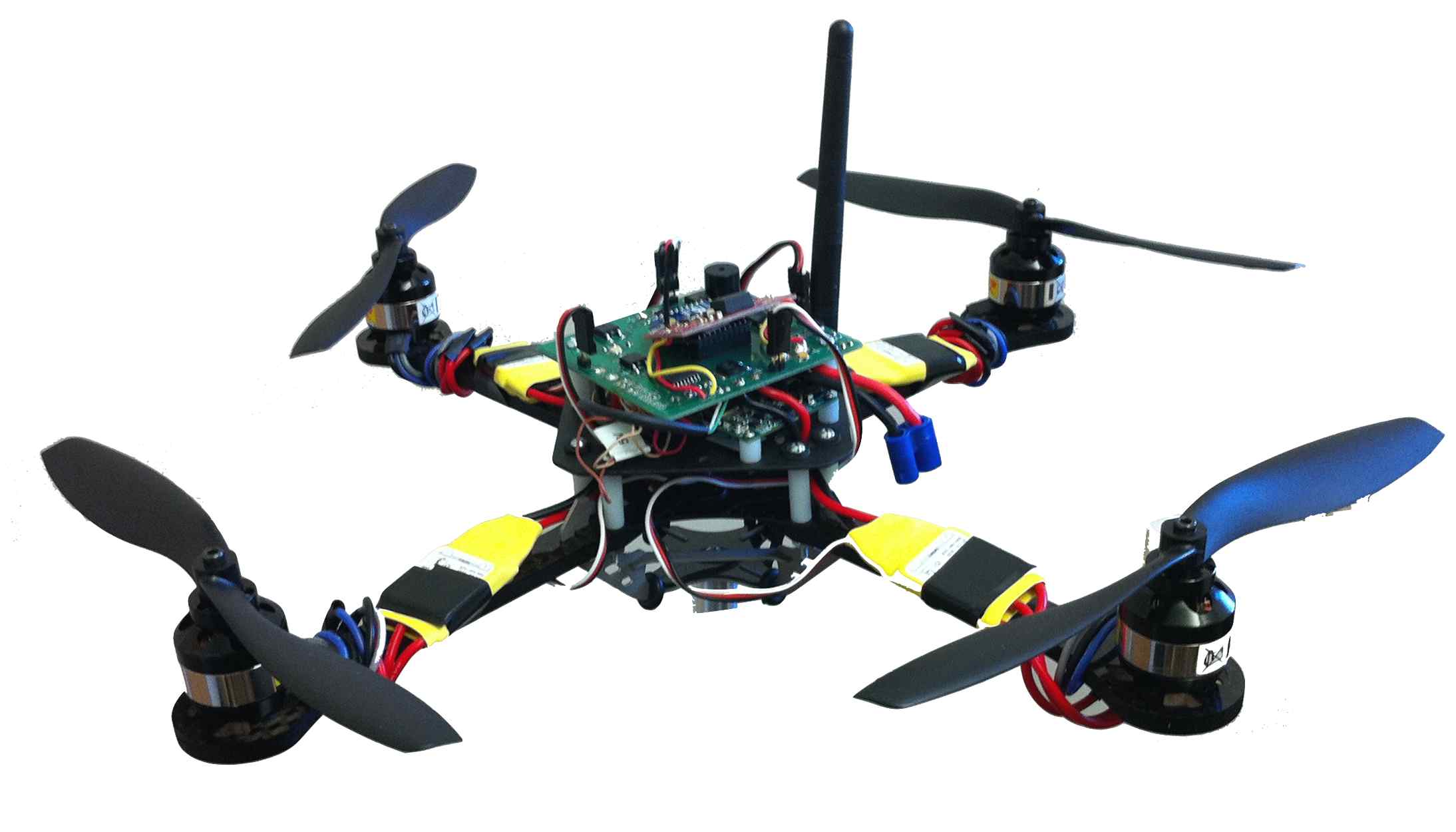}}
\put(1.95,3.2){\shortstack[c]{OMAP 600MHz\\Processor}}
\put(2.3,0){\shortstack[c]{Attitude sensor\\3DM-GX3\\ via UART}}
\put(0.85,1.4){\shortstack[c]{BLDC Motor\\ via I2C}}
\put(0.1,2.5){\shortstack[c]{Safety Switch\\XBee RF}}
\put(4.3,3.2){\shortstack[c]{WIFI to\\Ground Station}}
\put(5,2.0){\shortstack[c]{LiPo Battery\\11.1V, 2200mAh}}
\end{picture}}
%	\subfigure[Experiment Environment]{
%	\includegraphics[width=0.27\columnwidth]{quad_exp.png}}
}
\caption{Hardware development for a quadrotor UAV}\label{fig:Quad}
\end{figure}
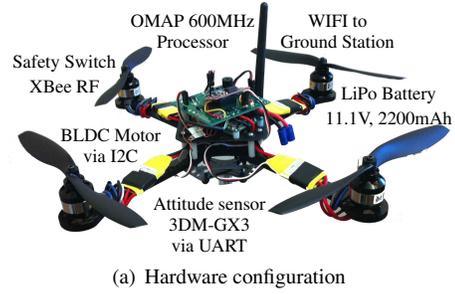

\begin{figure}
\centerline{
%	\subfigure[Attitude control testbed]{
	\includegraphics[width=1.0\columnwidth]{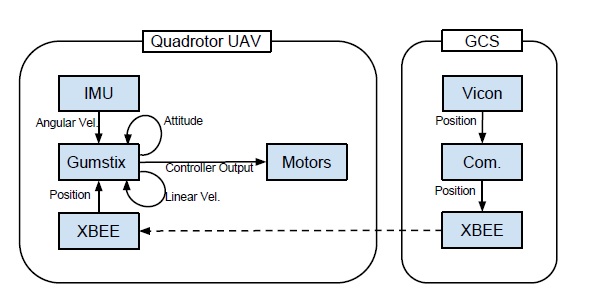}}
\caption{Information flow of overall system}\label{fig:information_flow}
\end{figure}

\subsection{Lissajous Curve Trajectory Tracking}
We consider tracking of arbitrary trajectories. The following Lissajous path is chosen as desired trajectory:
\begin{align*}
x_d(t) = \begin{cases}
x_o-\frac{t}{8}(x_o-x_i) &\hspace{-0.57cm} \mbox{if } 0 \leqq t < 8\\
[\sin(t-8)+\frac{\pi}{2}),\;\sin 2(t-8),\; -1.5]\,\mathrm{m} & \mbox{if } 8 \leqq t\\
\end{cases}.
\end{align*} 
The quadrotor takes-off from $x_o=[0.2,\,-2.8,\,-1.2]\mathrm{m}$ at $t=0\;\mathrm{sec}$ and flies to the initial position of the Lissajous curve trajectory where is $x_i=[1,\,0,\,1.5]\mathrm{m}$ by tracking a linear desired trajectory. Then, the quadrotor starts to follow the Lissajous curve trajectory at $t=8\;\mathrm{sec}$. There is about $0.15\;\mathrm{sec}$ of time delay from the Vicon motion capture system to the Gumstix. However, due to the robustness and stability properties of the proposed controller, position tracking performance shows satisfactory results as shown at Figure~\ref{fig:Lissajous_position} and \ref{fig:Lissajous_xyzz}.
%and it causes error chattering  during the trajectory tracking maneuver ($8-30\;\mathrm{sec}$) as shown in \ref{fig:Lissajous_error} unlikely to the hovering test.
\begin{figure}
\centerline{
	\subfigure[Attitude error variables $\Psi,e_R,e_\Omega$]{
		\includegraphics[width=0.55\columnwidth]{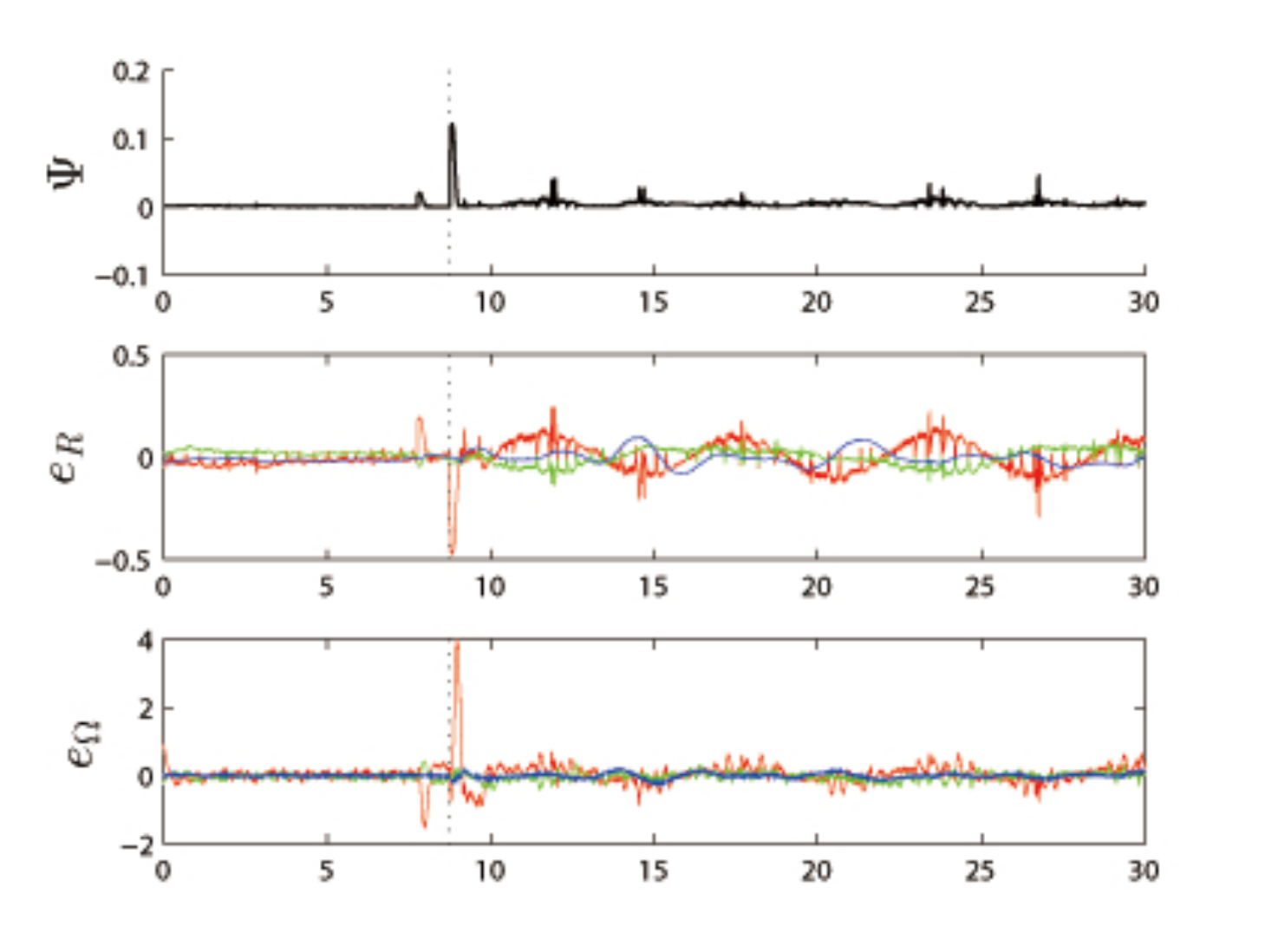}\label{fig:Lissajous_error}}
		%\hfill
		\hspace{-0.5cm}
			\subfigure[Thrust of each rotor ($\mathrm{N}$)]{
				\includegraphics[width=0.55\columnwidth]{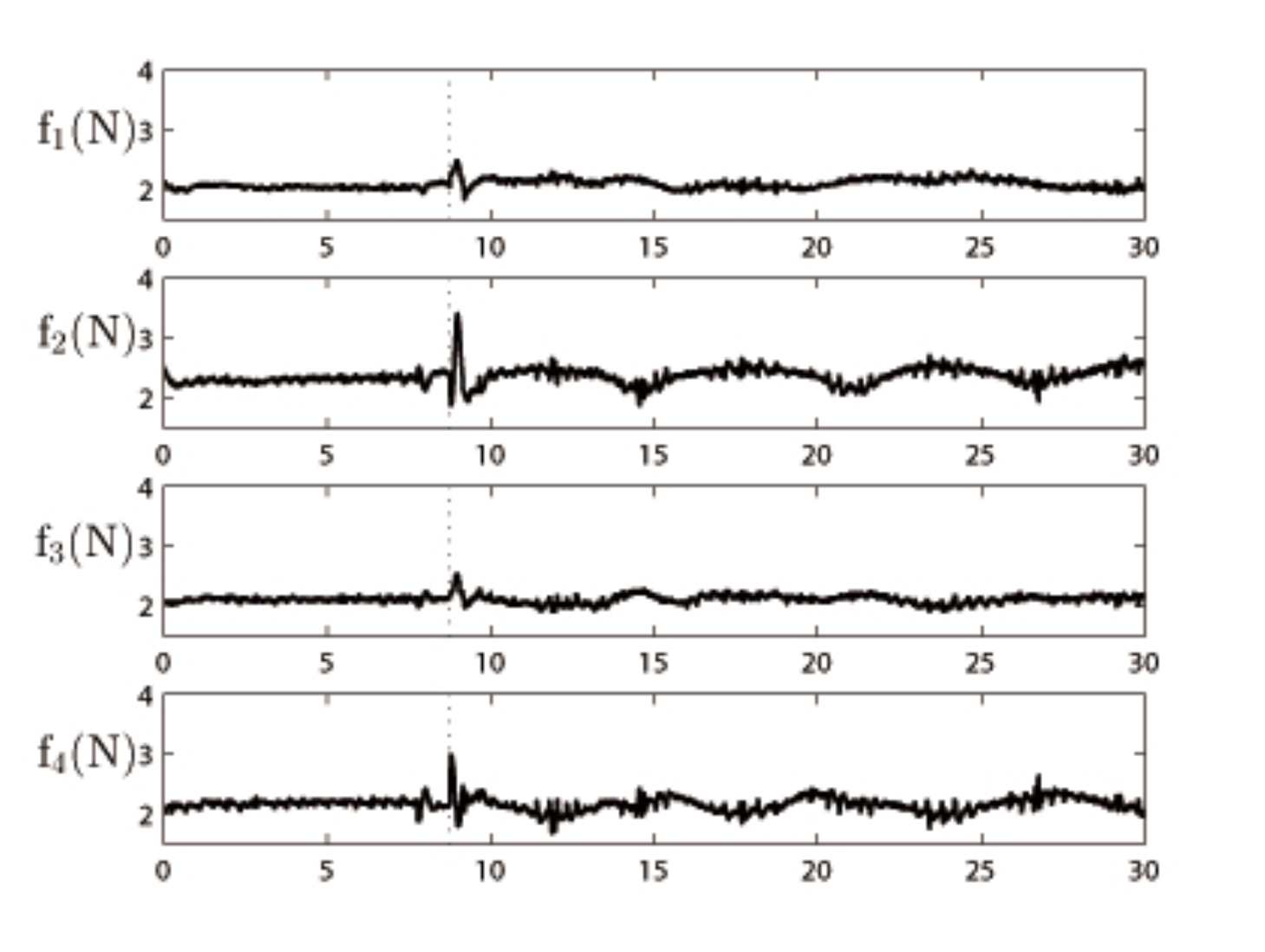}}
}
\centerline{
	\subfigure[Position (solid line) and desired (dotted line) $x,x_d$ ($\mathrm{m}$)]{
		\includegraphics[width=0.55\columnwidth]{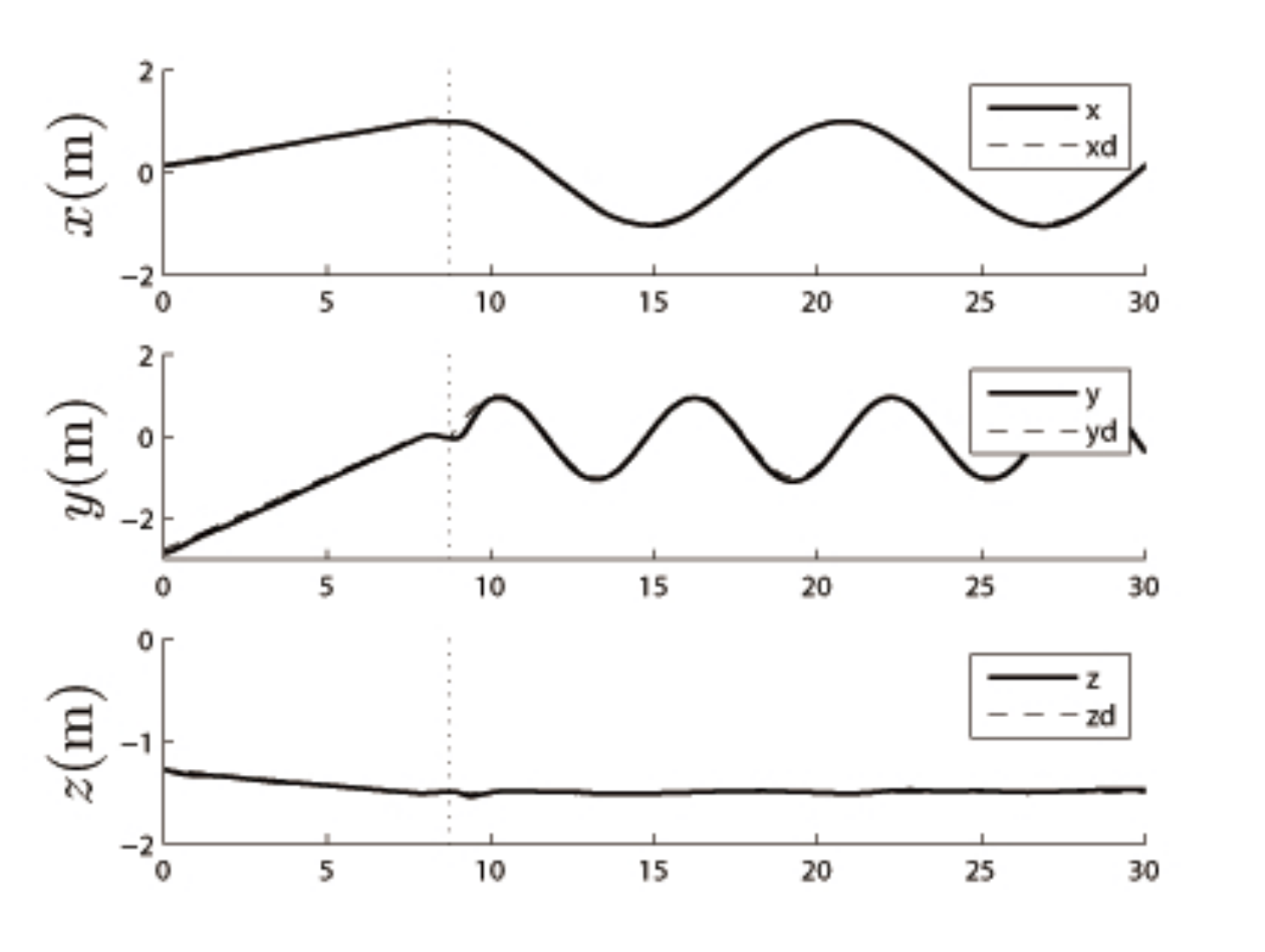}\label{fig:Lissajous_position}}
		%\hfill
		\hspace{-0.5cm}
			\subfigure[Linear velocity ($\mathrm{m/sec}$)]{
				\includegraphics[width=0.55\columnwidth]{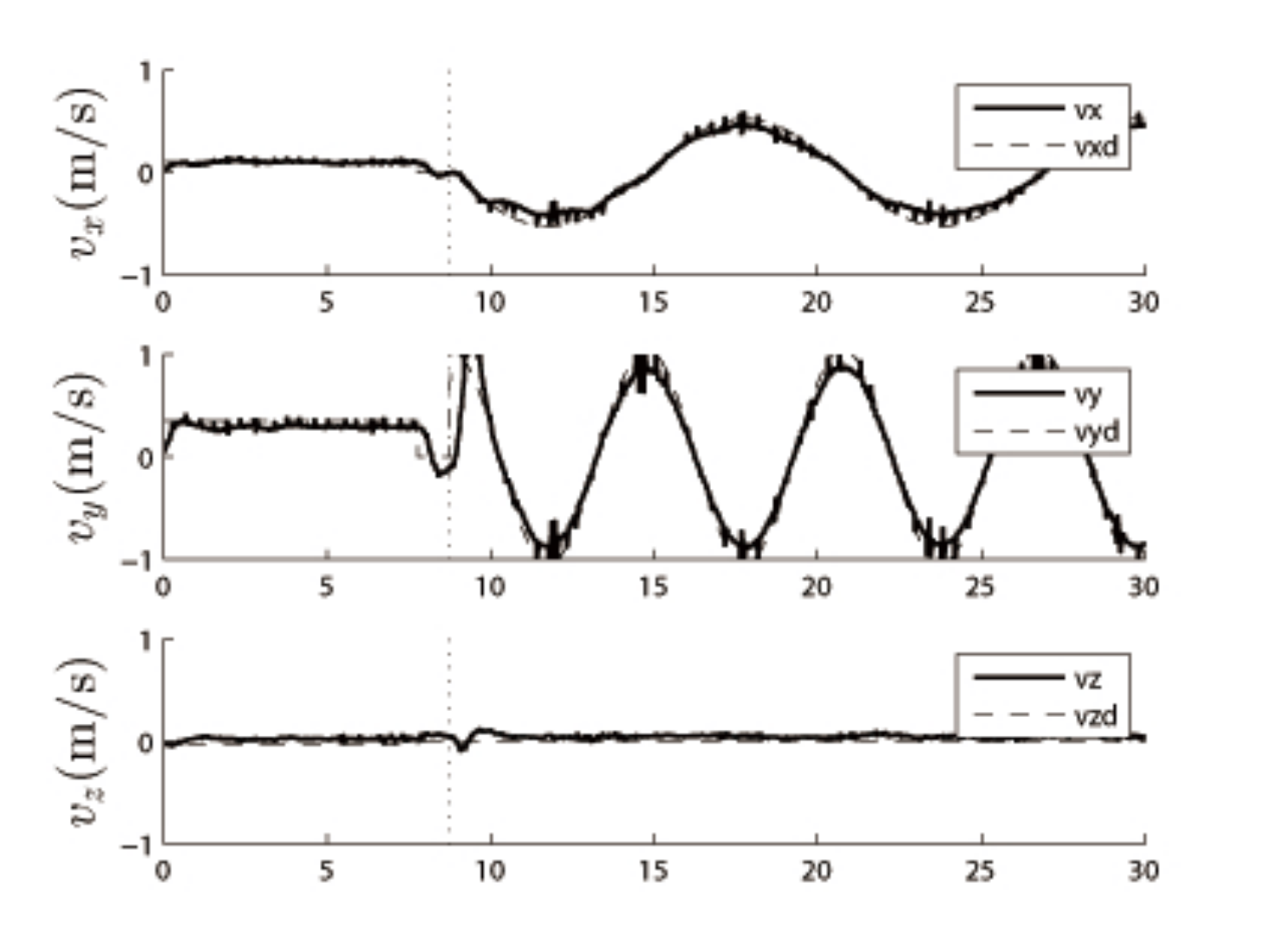}\label{fig:Lissajous_vel}}
}
\centerline{
	\subfigure[Eular angles ($\mathrm{rad}$)]{
		\includegraphics[width=0.55\columnwidth]{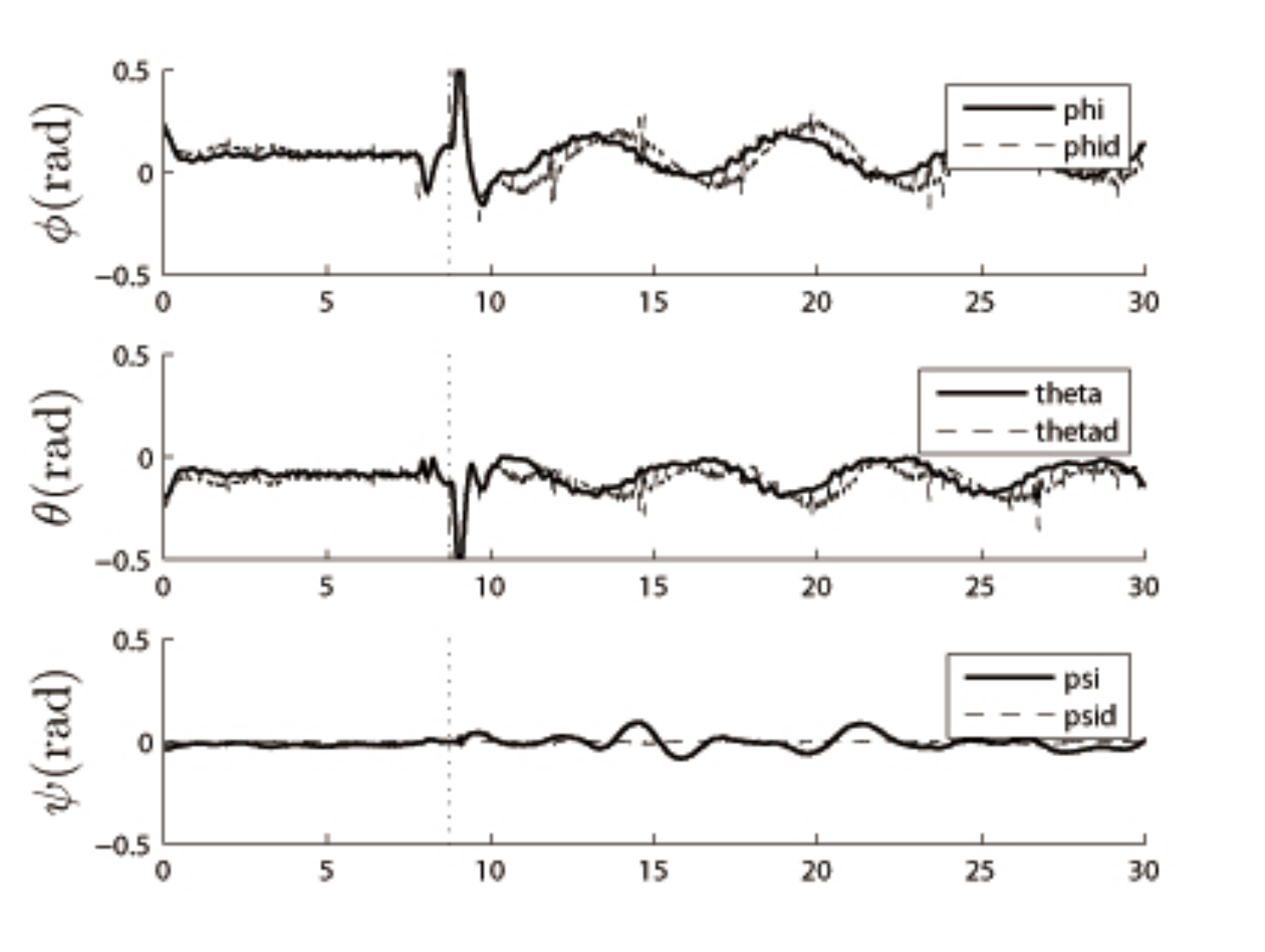}\label{fig:Lissajous_Eular}}
		%\hfill
		\hspace{-0.5cm}
			\subfigure[Angular velocity $\Omega,\Omega_d$ ($\mathrm{rad/sec}$)]{
				\includegraphics[width=0.55\columnwidth]{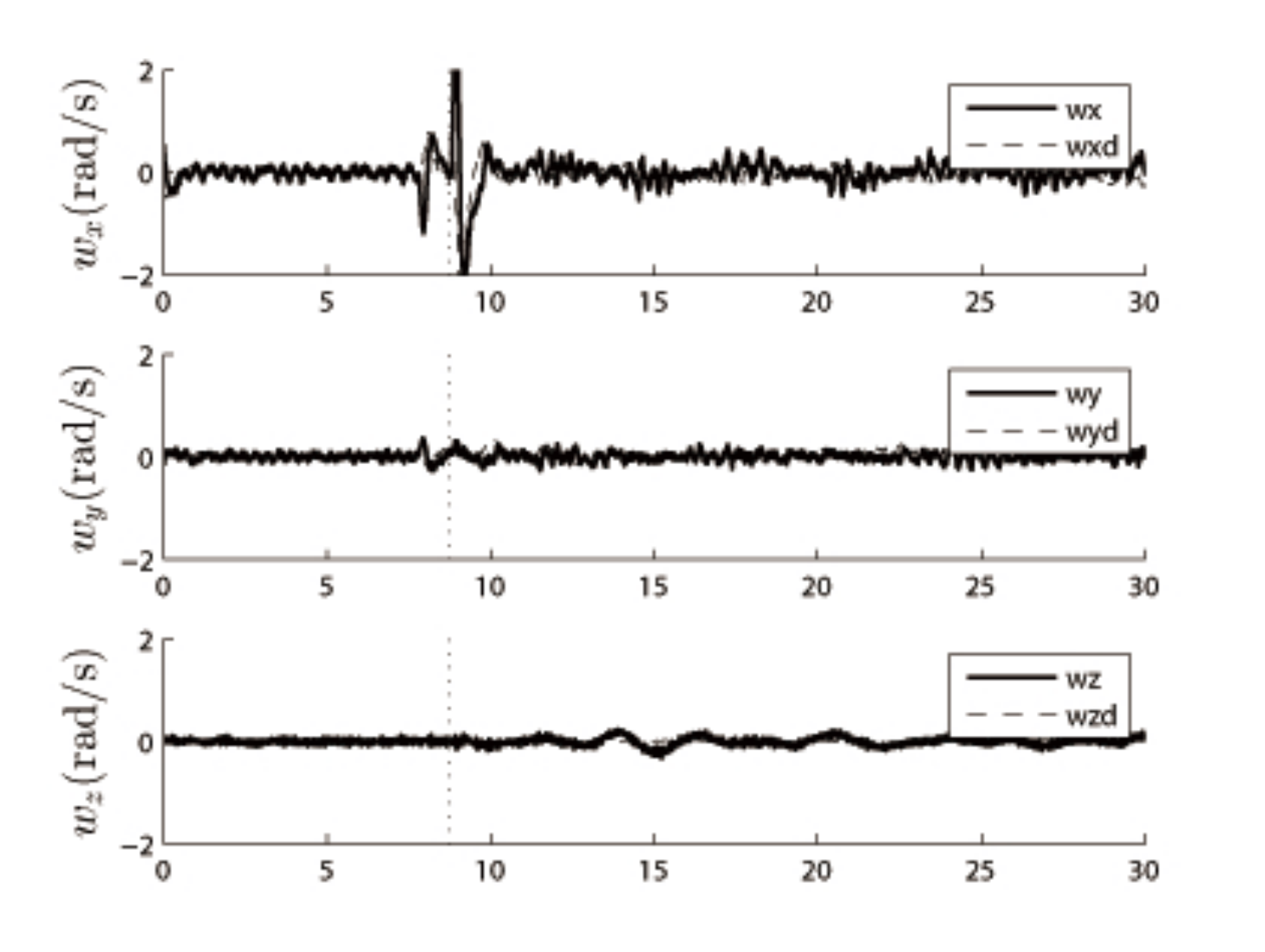}\label{fig:Lissajous_W}}
}
\caption{Lissajous curve trajectory tracking results (dotted:desired, solid:actual)}
\end{figure}

\begin{figure}
\centerline{		
\includegraphics[width=0.9\columnwidth]{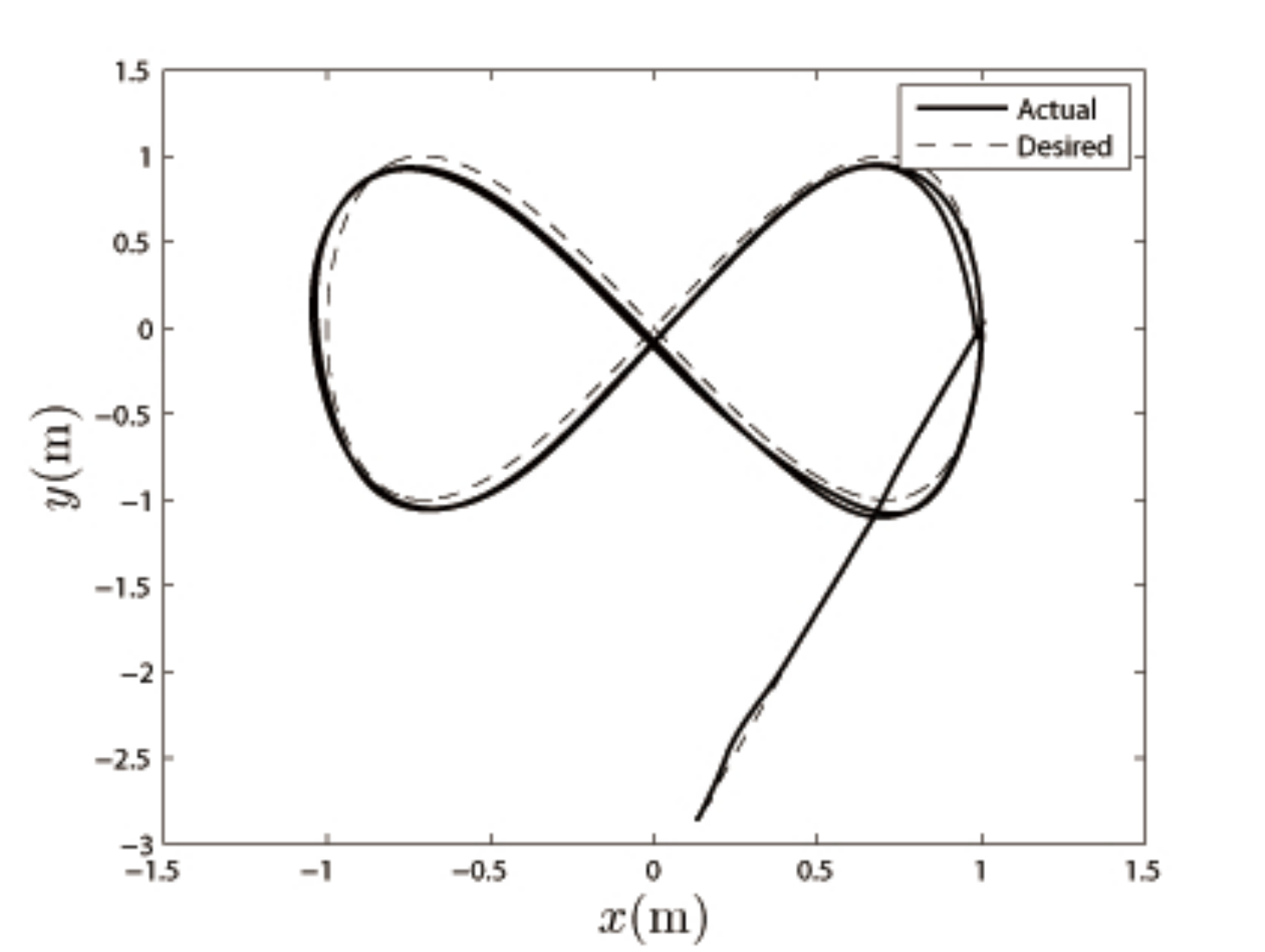}\label{fig:Lissajous_xyzzs}
}
\caption{Lissajous curve $x-y$ plane trajectory}\label{fig:Lissajous_xyzz}
\end{figure}

\subsection{Flipping}
Next, the proposed controller is validated with a flipping maneuver. The quadrotor takes off from a landing platform, increases altitude with constant speed to a constant point, flips $360$ degree about it $x$-axis. As presented in the numerical simulation section, this is a complex maneuver combining a nontrivial pitching maneuver with a yawing motion. It is achieved by concatenating the following two control modes of an attitude tracking same as presented in the numerical simulation to rotate the quadrotor
\begin{align*}
&R_d(t)= I+\sin(4 \pi t)\hat{e}_r+(1-\cos(4 \pi t))(e_r e_r^T-I),\\
&\Omega_d= 5\pi\cdot e_r.
\end{align*}
where $e_{r}=[1,\; 0,\;0]$, and a trajectory tracking mode to make it hover after completing the preceding rotation. As it is clear from the figures, the attitude control part which handles the rotation happens in almost $0.3$ seconds and then it switched to the position control mode to make the quadrotor stabilized and hovers to the desired position. Figure \ref{fig:ffgghhjjhhgg} and \ref{fig:snapflip} show the experimental results and snapshots of the flipping maneuver respectively. \footnote{A short video of the experiments is also available at \url{http://www.youtube.com/watch?v=wtn9L6BsYiE}.}
\begin{figure}\label{fig:ffssrrdd}
\centerline{
	\subfigure[Attitude error variables $\Psi,e_R,e_\Omega$]{
		\includegraphics[width=0.56\columnwidth]{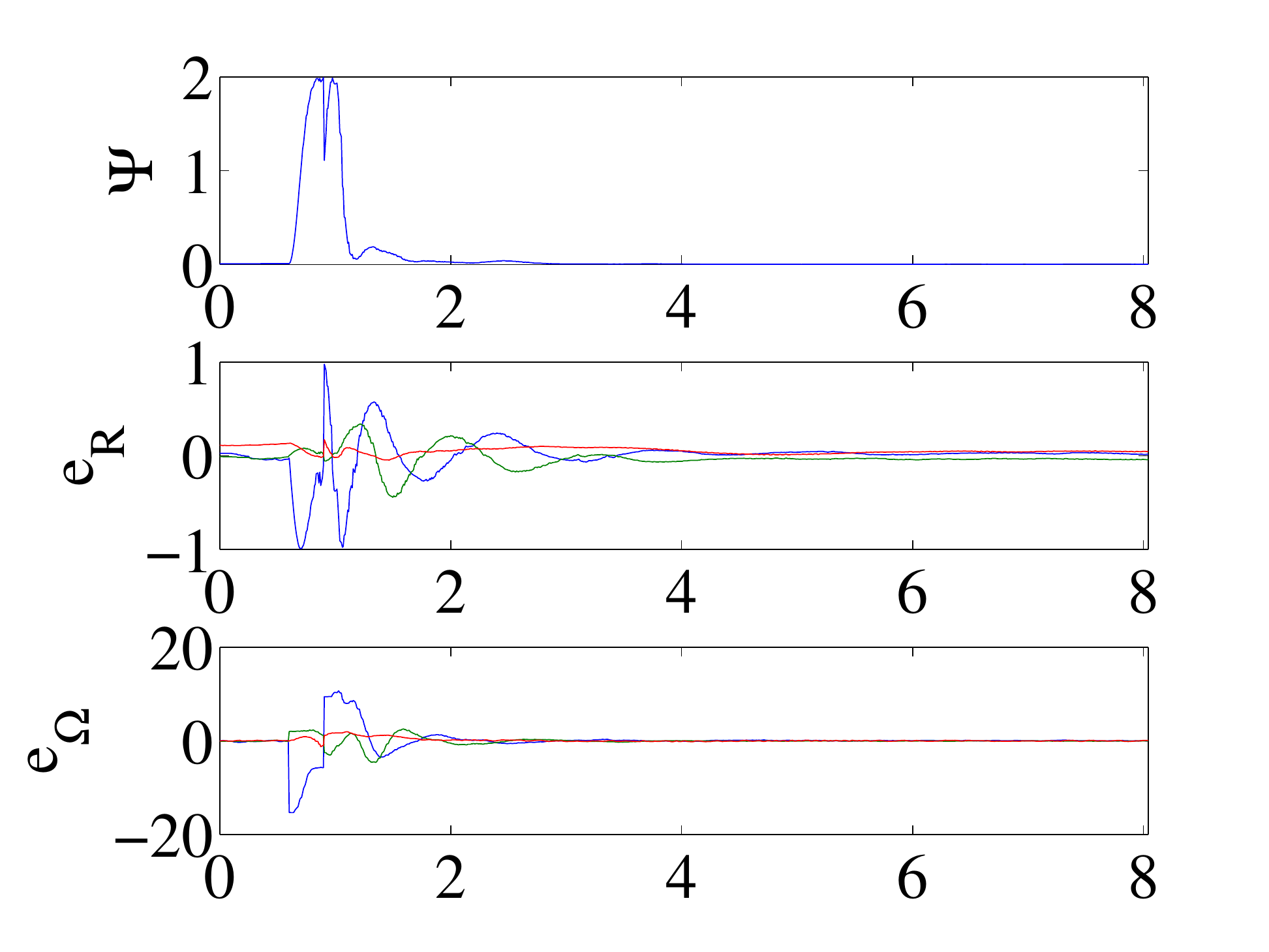}\label{fig:hover_error}}
		%\hfill
		\hspace{-0.5cm}
		\subfigure[Thrust of each rotor ($\mathrm{N}$)]{
				\includegraphics[width=0.55\columnwidth]{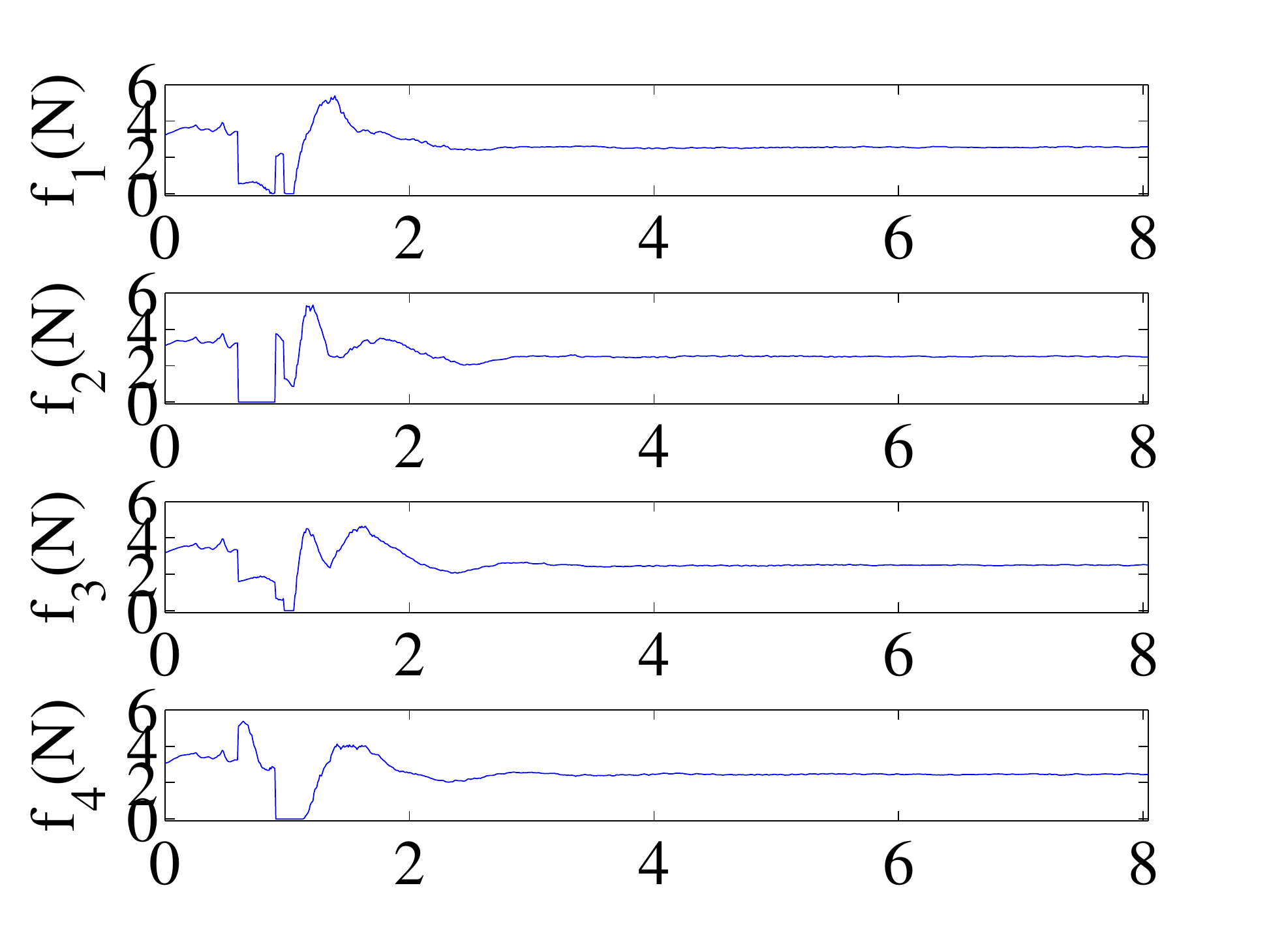}}
}
\centerline{
\subfigure[Position $x,x_d$ ($\mathrm{m}$)]{
		\includegraphics[width=0.55\columnwidth]{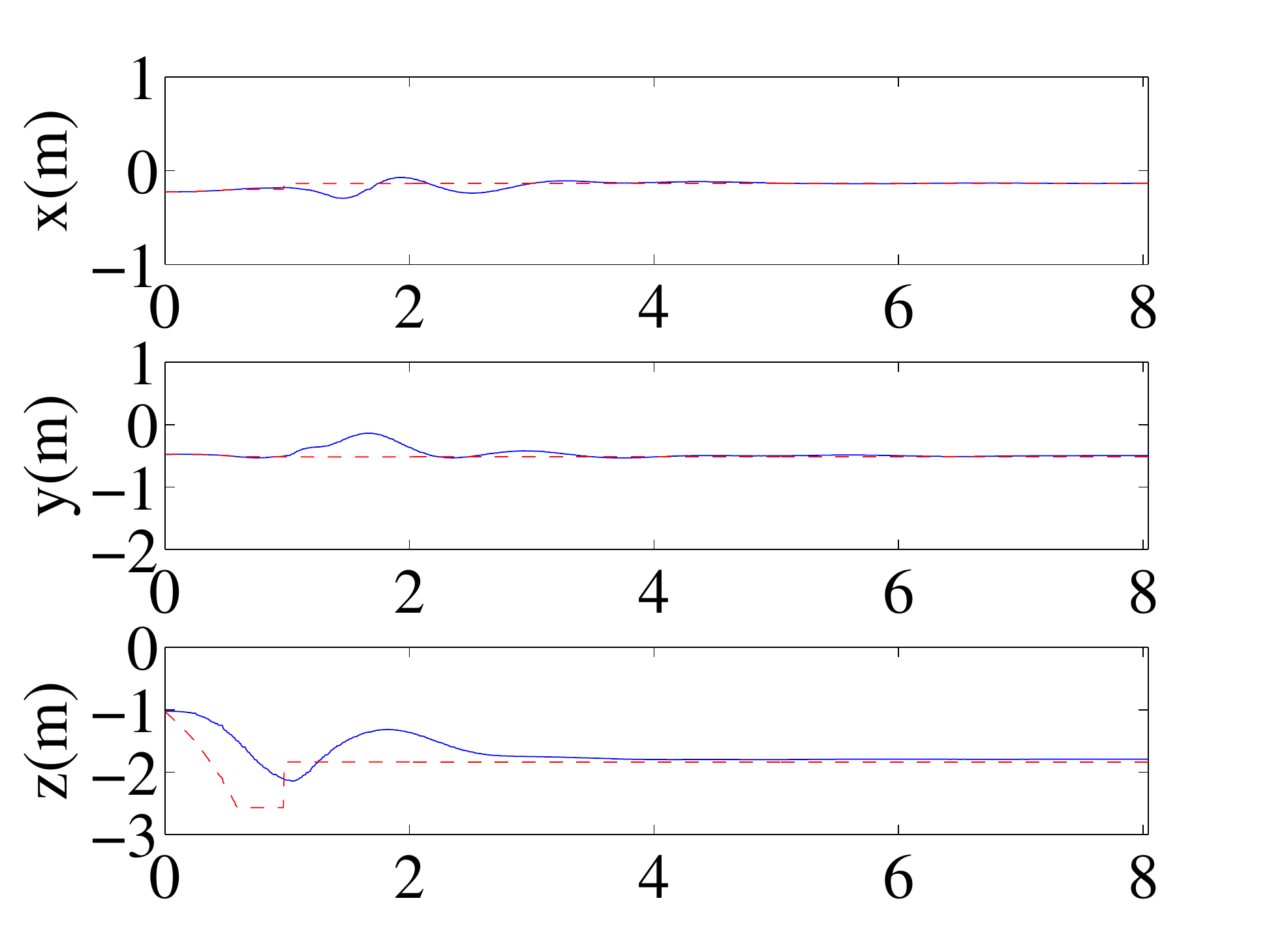}\label{fig:hover_position}}
		%\hfill
		\hspace{-0.5cm}
			\subfigure[Linear velocity ($\mathrm{m/sec}$)]{
				\includegraphics[width=0.55\columnwidth]{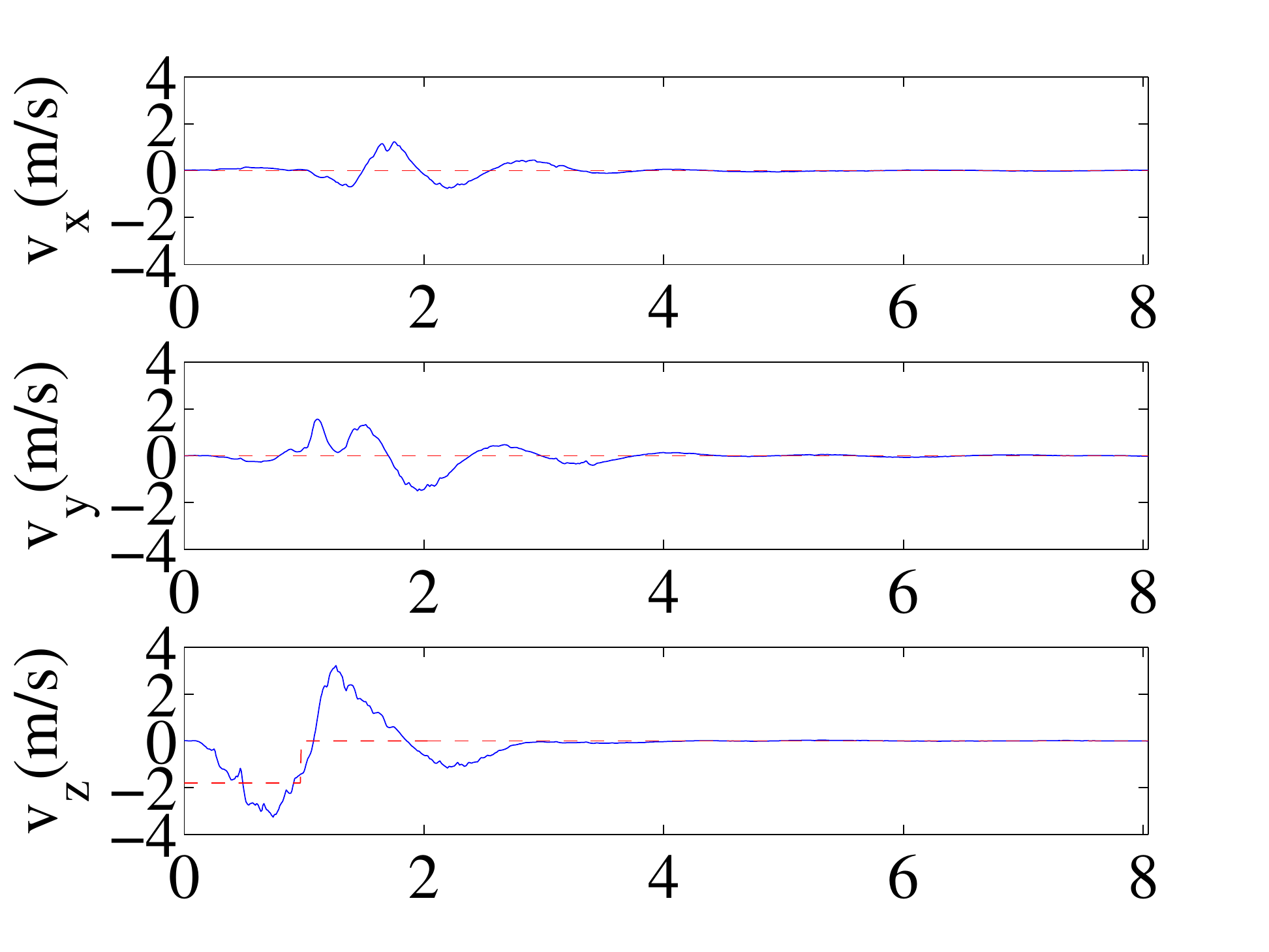}\label{fig:hover_vel}}
}
\centerline{
	\subfigure[Rotation Matrix ]{
		\includegraphics[width=0.55\columnwidth]{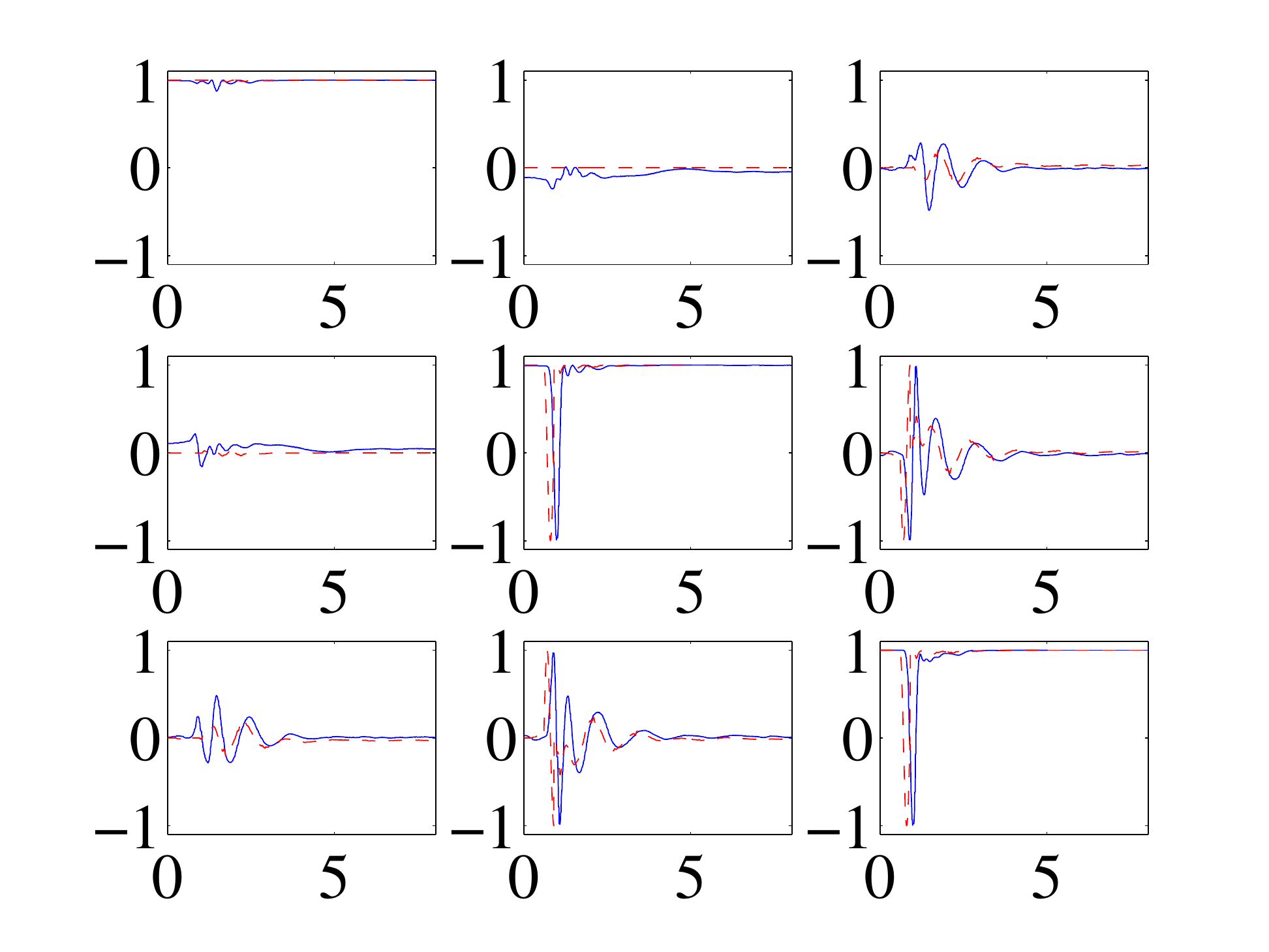}\label{fig:hover_Eular}}
		%\hfill
		\hspace{-0.5cm}
			\subfigure[Angular velocity $\Omega,\Omega_d$ ($\mathrm{rad/sec}$)]{
				\includegraphics[width=0.55\columnwidth]{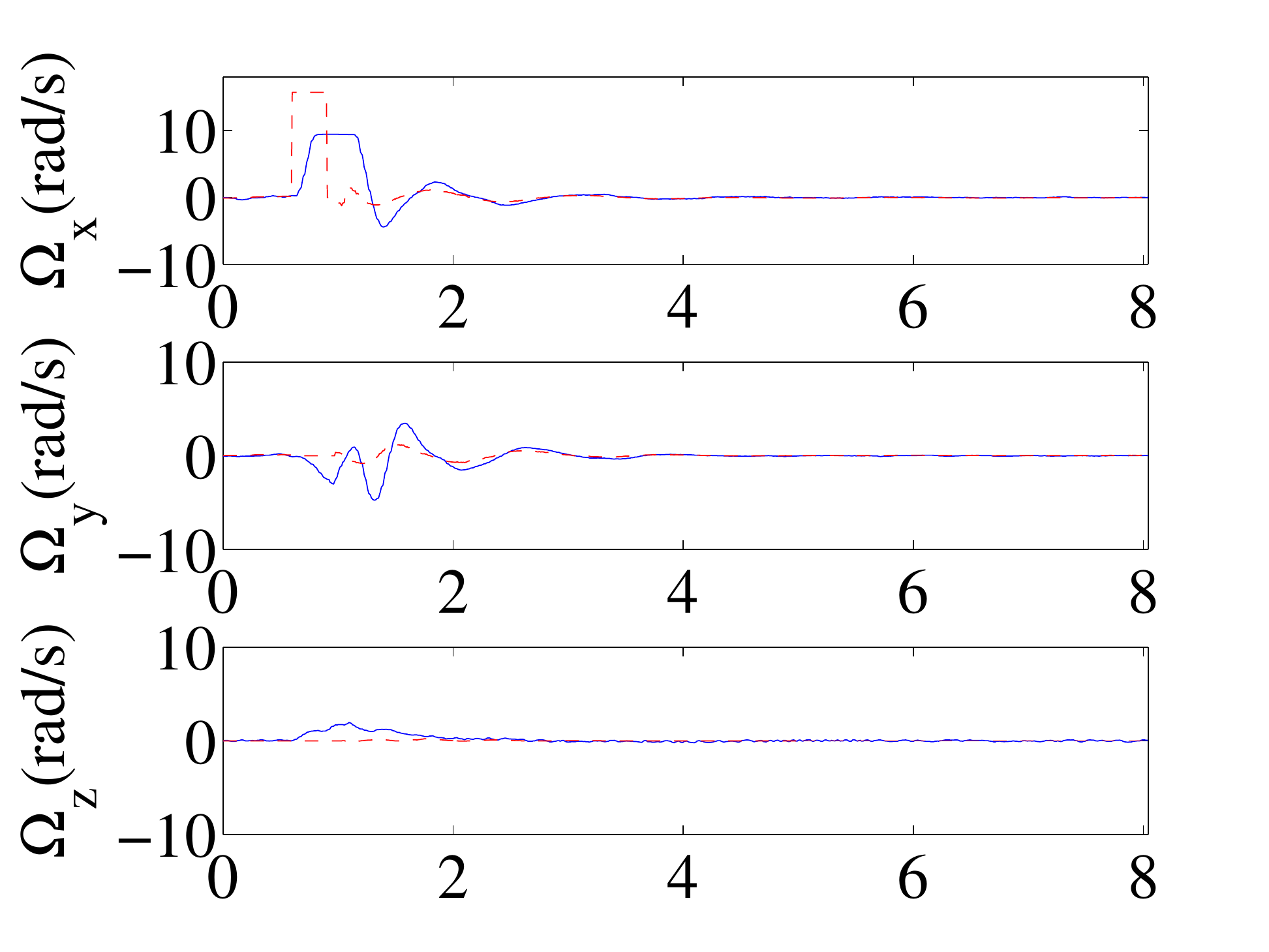}\label{fig:hover_W}}
}
\caption{Flipping flight test results (dotted:desired, solid:actual)}\label{fig:ffgghhjjhhgg}
\end{figure}
\begin{figure}
\centerline{
	\subfigure[$t=0.0$ sec]{
		\includegraphics[width=0.30\columnwidth]{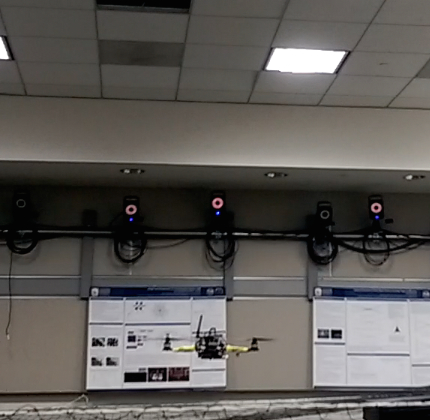}}
		\hfill
		\subfigure[$t=0.8756$ sec]{
				\includegraphics[width=0.30\columnwidth]{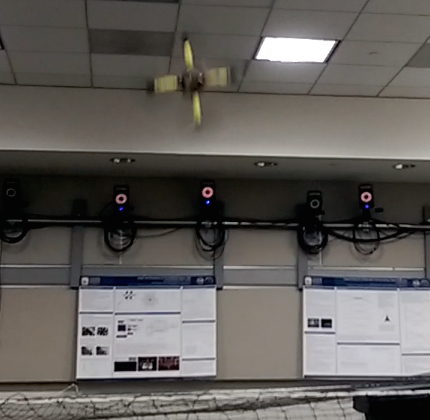}}
				\hfill
		\subfigure[$t=1.008$ sec]{
				\includegraphics[width=0.30\columnwidth]{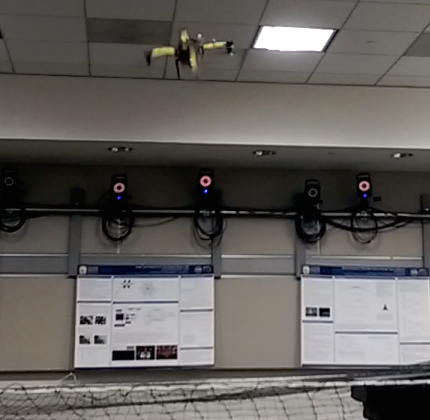}}
}
\centerline{
\subfigure[$t=1.079$ sec]{
		\includegraphics[width=0.30\columnwidth]{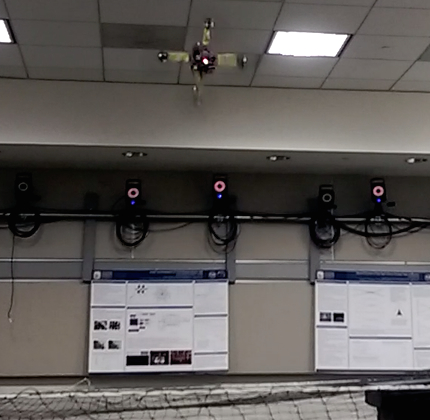}}
		\hfill
			\subfigure[$t=1.125$ sec]{
				\includegraphics[width=0.30\columnwidth]{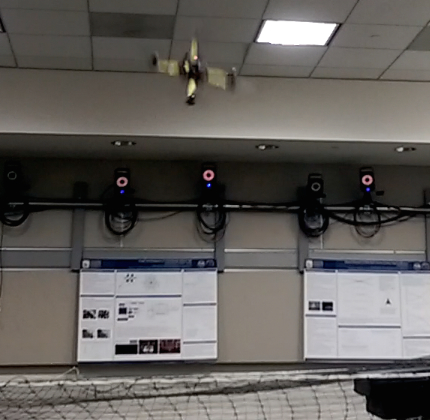}}
				\hfill
		\subfigure[$t=1.175$ sec]{
				\includegraphics[width=0.30\columnwidth]{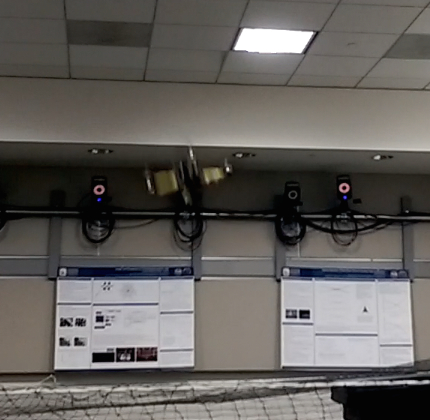}}
}
\centerline{
	\subfigure[$t=1.844$ sec]{
		\includegraphics[width=0.30\columnwidth]{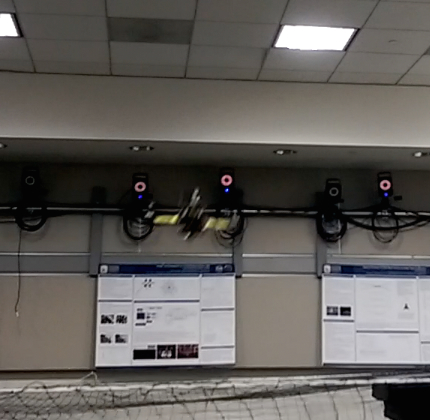}}
		\hfill
			\subfigure[$t=2.312$ sec]{
				\includegraphics[width=0.30\columnwidth]{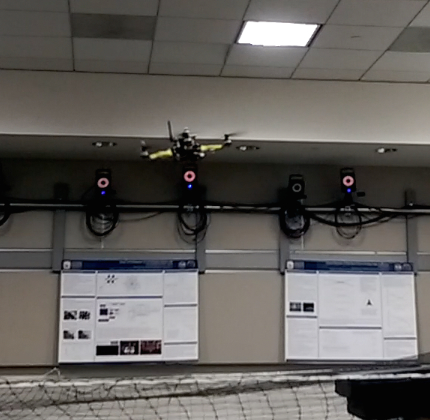}}
				\hfill
		\subfigure[$t=2.89$ sec]{
				\includegraphics[width=0.30\columnwidth]{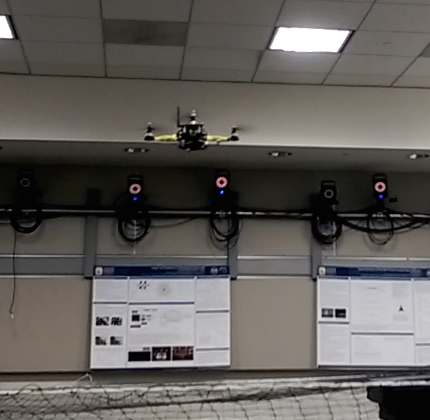}}
}
\caption{Snapshots for flipping maneuver.}
% \href{http://www.youtube.com/watch?v=wtn9L6BsYiE}{http://www.youtube.com/watch?v=wtn9L6BsYiE}
\label{fig:snapflip}
\end{figure}
%%%%%%%%%%%%%%%%%%%%%%%%%%%%%%%%%%%%%%%%%%%%%%%%%%%%%%%%%%%%%%%%%%%%%%
\section{Conclusions}
A new nonlinear adaptive control system is proposed for tracking control of quadrotor unmanned aerial vehicles. It is developed directly on the special orthogonal group to avoid complexities and ambiguities that are associated with Euler-angles or quaternions, and the proposed adaptive control term guarantees almost global attractivity for the tracking error variables in the existence of uncertainties. These are verified by rigorous mathematical analysis and experiments concurrently.

%%%%%%%%%%%%%%%%%%%%%%%%%%%%%%%%%%%%%%%%%%%%%%%%%%%%%%%%%%%%%%%%%%%%%%
\begin{acknowledgment}
This research has been supported in part by NSF under the grant CMMI-1243000 (transferred from 1029551), CMMI-1335008, and CNS-1337722.
\end{acknowledgment}

%%%%%%%%%%%%%%%%%%%%%%%%%%%%%%%%%%%%%%%%%%%%%%%%%%%%%%%%%%%%%%%%%%%%%%
% The bibliography is stored in an external database file
% in the BibTeX format (file_name.bib).  The bibliography is
% created by the following command and it will appear in this
% position in the document. You may, of course, create your
% own bibliography by using thebibliography environment as in
%
% \begin{thebibliography}{12}
% ...
% \bibitem{itemreference} D. E. Knudsen.
% {\em 1966 World Bnus Almanac.}
% {Permafrost Press, Novosibirsk.}
% ...
% \end{thebibliography}

% Here's where you specify the bibliography style file.
% The full file name for the bibliography style file 
% used for an ASME paper is asmems4.bst.
\bibliographystyle{asmems4}

% Here's where you specify the bibliography database file.
% The full file name of the bibliography database for this
% article is asme2e.bib. The name for your database is up
% to you.
\bibliography{asme2e}

%%%%%%%%%%%%%%%%%%%%%%%%%%%%%%%%%%%%%%%%%%%%%%%%%%%%%%%%%%%%%%%%%%%%%%
\appendix       %%% starting appendix
%\section*{Appendix: Properties and Proofs}
\section{Proof of Proposition \ref{prop:Att}}\label{sec:pfAtt}

We first find the error dynamics for $e_R,e_\Omega$, and define a Lyapunov function. Then, we find conditions on control parameters to guarantee the boundedness of tracking errors. Using \refeqn{EL3}, \refeqn{EL4}, \refeqn{M}, the time-derivative of $Je_\Omega$ can be written as
\begin{align}
J\dot e_\Omega & = \{Je_\Omega + d\}^\wedge e_\Omega - k_R e_R-k_\Omega e_\Omega+ \mathds{W}_{R}\tilde{\theta}_{R},\label{eqn:JeWdot}
\end{align}
where $d=(2J-\trs{J}I)R^TR_d\Omega_d\in\Re^3$ and  $\tilde{\theta}_{R}=\theta_{R}-\bar{\theta}_{R}$. The important property is that the first term of the right hand side is normal to $e_\Omega$, and it simplifies the subsequent Lyapunov analysis. Define a Lyapunov function $\mathcal{V}_2$ be 
\begin{align}
\mathcal{V}_2 & = \frac{1}{2} e_\Omega \cdot J e_\Omega + k_R\, \Psi(R,R_d)+c_2 e_R\cdot Je_\Omega+\frac{1}{2\gamma_{R}}\|\tilde{\theta}_{R}\|^2.\label{eqn:V2}
\end{align}
From \refeqn{PsiLB}, \refeqn{PsiUB}, the Lyapunov function $\mathcal{V}_2$ is bounded as
\begin{align}
z_2^T M_{21} z_2+\frac{1}{2\gamma_{R}}\|\tilde{\theta}_{R}\|^2
\leq \mathcal{V}_2 \leq z_2^T M_{22} z_2+\frac{1}{2\gamma_{R}}\|\tilde{\theta}_{R}\|^2,
\label{eqn:V2b}
\end{align}
where $z_2 =[\|e_R\|,\;\|e_\Omega\|]^T\in\Re^2$, and the matrices $M_{12},M_{22}$ are given by
\begin{align}
M_{21} = \frac{1}{2}\begin{bmatrix} k_R & -c_2\lambda_M \\ -c_2\lambda_M & \lambda_m  \end{bmatrix},\,
M_{22} = \frac{1}{2}\begin{bmatrix} \frac{2k_R}{2-\psi_2} & c_2\lambda_M \\ c_2\lambda_M & \lambda_{M}\end{bmatrix}.
%\label{eqn:M2}
\end{align}
From \refeqn{PsiUB}, the upper-bound of \refeqn{V2b} is satisfied in the following domain:
\begin{align}
D_2 = \{ (R,\Omega)\in \SO\times\Re^3\,|\, \Psi(R,R_d)<\psi_2<2\}.\label{eqn:D2}
\end{align}
From \refeqn{Psidot00}, \refeqn{JeWdot}, the time derivative of $\mathcal{V}_2$ along the solution of the controlled system is given by
\begin{align*}
\dot{\mathcal{V}}_2  =&
-k_\Omega\|e_\Omega\|^2  +e_{\Omega}^{T}\mathds{W}_{R}\tilde{\theta}_{R}\\
&+ c_2 \dot e_R \cdot Je_\Omega+ c_2 e_R \cdot J\dot e_\Omega +\frac{1}{\gamma_{R}}(\tilde{\theta})^{T}(\dot{\tilde{\theta}}).
\end{align*}
We have $\dot{\tilde{\theta}}_{R}=-\dot{\bar{\theta}}_{R}$. Substituting \refeqn{JeWdot}, the above equation becomes
\begin{align*}
\dot{\mathcal{V}}_2  =&
-k_\Omega\|e_\Omega\|^2  + c_2 \dot e_R \cdot Je_\Omega-c _2 k_R \|e_R\|^2 \\
&+ c_2 e_R \cdot ((Je_\Omega+d)^\wedge e_\Omega -k_\Omega e_\Omega)\\
&+\tilde{\theta}_{R}^{T}\mathds{W}_{R}^{T}(e_{\Omega}+c_{2}e_{R})-\frac{1}{\gamma_{R}}\tilde{\theta}_{R}^{T}\dot{\bar{\theta}}_{R}.
\end{align*}
By substituting the adaptive law given by \refeqn{eI}
\begin{align*}
\dot{\mathcal{V}}_2  =&
-k_\Omega\|e_\Omega\|^2  + c_2 \dot e_R \cdot Je_\Omega-c _2 k_R \|e_R\|^2 \\
&+ c_2 e_R \cdot ((Je_\Omega+d)^\wedge e_\Omega -k_\Omega e_\Omega).
\end{align*}
Since $\|e_R\|\leq 1$, $\|\dot e_R\|\leq \|e_\Omega\|$, and $\|d\|\leq B_2$, we have
\begin{align}
\dot{\mathcal{V}}_2 \leq - z_2^T W_2 z_2,\label{eqn:dotV2}
\end{align}
where the matrix $W_2\in\Re^{2\times 2}$ is given by
\begin{align*}
W_2 = \begin{bmatrix} c_2k_R & -\frac{c_2}{2}(k_\Omega+B_2) \\ 
-\frac{c_2}{2}(k_\Omega+B_2) & k_\Omega-2c_2\lambda_M \end{bmatrix}.%\label{eqn:W2}
\end{align*}
%The condition on $c_2$ given at \refeqn{c2} guarantees that all of matrices $M_{21}$, $M_{22}$, $W_2$ are positive definite. This implies that the zero equilibrium of tracking errors $(e_R,e_\Omega)=(0,0)$ is stable in the sense of Lyapunov, and $e_R,e_\Omega\rightarrow 0$ as $t\rightarrow\infty$, and furthermore $\tilde{\theta}_{R}$ is uniformly bounded.

\section{Proof of Proposition \ref{prop:Pos}}\label{sec:pfPos}
\setcounter{paragraph}{0}
We derive the tracking error dynamics and a Lyapunov function for the translational dynamics of a quadrotor UAV, and later it is combined with the stability analyses of the rotational dynamics. The subsequent analyses are developed in the domain $D_1$
\begin{align}
D_1=\{&(e_x,e_v,R,e_\Omega)\in\Re^3\times\Re^3\times \SO\times\Re^3\,|\,\nonumber\\
& \|e_x\|< e_{x_{\max}},\;\Psi< \psi_1 < 1\},\label{eqn:D}
\end{align}
Similar to \refeqn{PsiUB}, we can show that 
\begin{align}
\frac{1}{2} \norm{e_R}^2 \leq  \Psi(R,R_c) \leq \frac{1}{2-\psi_1} \norm{e_R}^2\label{eqn:eRPsi1}.
\end{align}

\subsection{Translational Error Dynamics} The time derivative of the position error is $\dot e_x=e_v$. The time-derivative of the velocity error is given by
\begin{align}
m\dot e_v = m\ddot x -m\ddot x_d = mg e_3 - fRe_3 -m\ddot x_d+\mathds{W}_{x}\theta_{x}. \label{eqn:evdot0}
\end{align}
Consider the quantity $e_3^T R_c^T R e_3$, which represents the cosine of the angle between $b_3=Re_3$ and $b_{3_c}=R_ce_3$. Since $1-\Psi(R,R_c)$ represents the cosine of the eigen-axis rotation angle between $R_c$ and $R$, we have $e_3^T R_c^T R e_3\geq  1-\Psi(R,R_c)>0$ in $D_1$. Therefore, the quantity $\frac{1}{e_3^T R_c^T R e_3}$ is well-defined. To rewrite the error dynamics of $e_v$ in terms of the attitude error $e_R$, we add and subtract $\frac{f}{e_3^T R_c^T R e_3}R_c e_3$ to the right hand side of \refeqn{evdot0} to obtain
\begin{align}
m\dot e_v &  = mg e_3 -m\ddot x_d- \frac{f}{e_3^T R_c^T R e_3}R_c e_3 - X+\mathds{W}_{x}\theta_{x},\label{eqn:evdot1}
\end{align}
where $X\in\Re^3$ is defined by
\begin{align}
X=\frac{f}{e_3^T R_c^T R e_3}( (e_3^T R_c^T R e_3)R e_3 -R_ce_3).\label{eqn:X}
\end{align}
Let $A=-k_x e_x - k_v e_v -\mathds{W}_{x}\bar{\theta}_{x}-mg e_3 + m\ddot x_d$.
%be the desired control force for the translational dynamics. 
Then, from \refeqn{Rd3}, \refeqn{f}, we have ${b}_{3_c}=R_c e_3 = -A/\norm{A}$ and $f=-A\cdot Re_3$. By combining these, we obtain $f= (\norm{A}R_c e_3)\cdot R e_3$. Therefore, the third term of the right hand side of \refeqn{evdot1} can be written as
\begin{align*}
-  \frac{f}{e_3^T R_c^T R e_3} & R_c e_3 = -\frac{(\norm{A}R_c e_3)\cdot R e_3}{e_3^T R_c^T R e_3}\cdot - \frac{A}{\norm{A}}=A\\
& =-k_x e_x - k_v e_v -\mathds{W}_{x}\bar{\theta}_{x} -mg e_3 + m\ddot x_d.
\end{align*}
Substituting this into \refeqn{evdot1}, the error dynamics of $e_v$ can be written as
\begin{align}
m\dot e_v  = & -k_x e_x - k_v e_v -\mathds{W}_{x}\bar{\theta}_{x} - X+\mathds{W}_{x}{\theta}_{x}\nonumber \\
&=  -k_x e_x - k_v e_v +\mathds{W}_{x}\tilde{\theta}_{x} - X.\label{eqn:evdot}
\end{align}
where $\tilde{\theta}_{x}=\theta_{x}-\bar{\theta}_{x}$ is the estimation errors.

\subsection{Lyapunov Candidate for Translation Dynamics}
Let a Lyapunov candidate $\mathcal{V}_1$ be
\begin{align}
\mathcal{V}_1 & = \frac{1}{2}k_x\|e_x\|^2  + \frac{1}{2} m \|e_v\|^2 + c_1 e_x\cdot me_v+ \frac{1}{2\gamma_{x}}\|\tilde{\theta}_{x}\|^2
\label{eqn:V1}.
\end{align}
The derivative of ${\mathcal{V}}_1$ along the solution of \refeqn{evdot} is given by
\begin{align}
\dot{\mathcal{V}}_1 
 =&  -(k_v-mc_1) \|e_v\|^2 
- c_1 k_x \|e_x\|^2 \nonumber\\
&-c_1 k_v e_x\cdot e_v+X\cdot \braces{ c_1 e_x + e_v}\nonumber\\
&+\mathds{W}_{x}\tilde{\theta}_{x}\cdot\{e_{v}+c_{1}e_{x}\}-\frac{1}{\gamma_{x}}\tilde{\theta}_{x}^{T}\dot{\bar{\theta}}_{x}.\label{eqn:V1dot03}
\end{align}
For the first case of the adaptive law given by \refeqn{adaptivelawx}, the last two terms of \refeqn{V1dot03} are cancelled out. Also for the second case~\cite{Ioannou96} that $\dot{\bar{\theta}}_{x}=\gamma_{x}(I-\frac{\bar{\theta}_{x}\bar{\theta}_{x}^{T}}{\bar{\theta}_{x}^{T}\bar{\theta}_{x}})\mathds{W}_{x}^{T}(e_{v}+c_{1}e_{x})$, we obtain
\begin{align}
\dot{\mathcal{V}}_1 
 =&  -(k_v-mc_1) \|e_v\|^2 
- c_1 k_x \|e_x\|^2 \nonumber\\
&-c_1 k_v e_x\cdot e_v+X\cdot \braces{ c_1 e_x + e_v}\nonumber\\
&+\tilde{\theta}_{x}^{T}\frac{\bar{\theta}_{x}\bar{\theta}_{x}^{T}}{\bar{\theta}_{x}^{T}\bar{\theta}_{x}}W_{x}^{T}(e_{v}+c_{1}e_{x}).
\end{align}
In the above equation, the last term on the right hand side is always negative since $\bar{\theta}_{x}^{T}\mathds{W}_{x}^T(e_{v}+c_{1}e_{x})>0$ and $\tilde{\theta}_{x}^T\bar{\theta}_{x}\leq 0$. Therefore for both cases of \refeqn{adaptivelawx}, we obtain
\begin{align}\label{eqn:V1dot0}
\dot{\mathcal{V}}_1 
 \leq&  -(k_v-mc_1) \|e_v\|^2 
- c_1 k_x \|e_x\|^2 \nonumber\\
&-c_1 k_v e_x\cdot e_v+X\cdot \braces{ c_1 e_x + e_v}.
\end{align}
The last term of the above equation corresponds to the effects of the attitude tracking error on the translational dynamics. We find a bound of $X$, defined at \refeqn{X}, to show stability of the coupled translational dynamics and rotational dynamics in the subsequent Lyapunov analysis. Since $f=\|A\| (e_3^T R_c^T R e_3)$, we have
\begin{align*}
\norm{X}  \leq& \|A\|\,\| (e_3^T R_c^T R e_3)R e_3 -R_ce_3\|\\
 \leq&( k_x \|e_x\| + k_v \|e_v\| + B_{W_{x}} B_{\theta} +B_1) \\
&\times \| (e_3^T R_c^T R e_3)R e_3 -R_ce_3\|.
\end{align*}
The last term $\| (e_3^T R_c^T R e_3)R e_3 -R_ce_3\|$ represents the sine of the angle between $b_3=Re_3$ and $b_{c_3}=R_c e_3$, since $(b_{3_c}\cdot b_3)b_3 - b_{3_c} = b_{3}\times (b_3\times b_{3_c})$. The magnitude of the attitude error vector, $\|e_R\|$ represents the sine of the eigen-axis rotation angle between $R_c$ and $R$ (see \cite{LeeLeoPICDC10}). Therefore, $\| (e_3^T R_c^T R e_3)R e_3 -R_ce_3\| \leq \| e_R\|$ in $D_1$. It follows that 
\begin{align}
\| (e_3^T R_d^T R e_3)R e_3 & -R_de_3\| \leq \| e_R\| = \sqrt{\Psi(2-\Psi)}\nonumber\\
& \leq \braces{\sqrt{\psi_1 (2-\psi_1)}\triangleq\alpha}  <1.\label{eqn:eR_bound}
\end{align}
Therefore, $X$ is bounded by
\begin{align}
\norm{X} 
&\leq ( k_x \|e_x\| + k_v \|e_v\| + B_{W_{x}} B_{\theta} + B_1) \|e_R\| \nonumber\\
&\leq ( k_x \|e_x\| + k_v \|e_v\| + B_{W_{x}} B_{\theta} + B_1) \alpha.\label{eqn:XB}
\end{align}
Substituting \refeqn{XB} into \refeqn{V1dot0}, 
\begin{align}
\dot{\mathcal{V}}_1 
& \leq   -(k_v(1-\alpha)-mc_1) \|e_v\|^2 
- {c_1 k_x}(1-\alpha) \|e_x\|^2 \nonumber\\
& + {c_1k_v}(1+\alpha) \|e_x\|\|e_v\|\nonumber\\
& +  \|e_R\| \braces{(B_{W_{x}} B_{\theta} +B_1)({c_1} \|e_x\| + \|e_v\|)+k_x\|e_x\|\|e_v\|}.\label{eqn:V1dot1}
\end{align}
In the above expression for $\dot{\mathcal{V}}_1$, there is a third-order error term, namely $k_x\|e_R\|\|e_x\|\|e_v\|$. Using \refeqn{eR_bound}, it is possible to choose its upper bound as $k_x\alpha\|e_x\|\|e_v\|$ similar to other terms, but the corresponding stability analysis becomes complicated, and the initial attitude error should be reduced further. Instead, we restrict our analysis to the domain $D_1$ defined in \refeqn{D}, and its upper bound is chosen as $k_xe_{x_{\max}}\|e_R\|\|e_v\|$.
\subsection{Lyapunov Candidate for the Complete System}
Let $\mathcal{V}=\mathcal{V}_1+\mathcal{V}_2$ be the Lyapunov candidate of the complete system. Define $z_1=[\|e_x\|,\;\|e_v\|]^T$, $z_2=[\|e_R\|,\;\|e_\Omega\|]^T\in\Re^2$, and 
\begin{align*}
\mathcal{V}_A = \frac{1}{2\gamma_{x}}\|\theta_{x}-\bar{\theta}_{x}\|^2+\frac{1}{2\gamma_{R}}\|\theta_{R}-\bar{\theta}_{R}\|^2.
\end{align*}
Using \refeqn{eRPsi1}, the bound of the Lyapunov candidate $\mathcal{V}$ can be written as
\begin{align}
z_1^T M_{11} z_1 + z_2^T M_{21} z_2& + \mathcal{V}_A
 \leq \mathcal{V} \leq z_1^T M_{12} z_1 + z_2^T M'_{22} z_2 +\mathcal{V}_A,\label{eqn:Vb}
\end{align}
where the matrices $M_{11},M_{12},M_{21},M_{22}$ are given by
\begin{gather*}
M_{11} = \frac{1}{2}\begin{bmatrix} k_x & -mc_1 \\ -mc_1 & m\end{bmatrix},\;
M_{12} = \frac{1}{2}\begin{bmatrix} k_x & mc_1 \\ mc_1 & m\end{bmatrix},\\
M_{21} = \frac{1}{2}\begin{bmatrix} k_R & -c_2\lambda_M \\ -c_2\lambda_M & \lambda_{m}  \end{bmatrix},\;
M_{22} = \frac{1}{2}\begin{bmatrix} \frac{2k_R}{2-\psi_1} & c_2\lambda_M \\ c_2\lambda_M & \lambda_{M}\end{bmatrix}.
\end{gather*}
Using \refeqn{dotV2} and \refeqn{V1dot1}, the time-derivative of $\mathcal{V}$ is given by
\begin{align}
\dot{\mathcal{V}} & \leq -z_1^T W_1 z_1  + z_1^T W_{12} z_2 - z_2^T W_2 z_2\leq -z^T W z
\end{align}
where $z=[z_1,z_2]^T\in\Re^2$, and the matrices $W_1,W_{12},W_2\in\Re^{2\times 2}$ are defined at \refeqn{W1}-\refeqn{W2}. The matrix $W\in\Re^{2\times 2}$ is given by
\begin{align*}
W=\begin{bmatrix}
\lambda_{m}(W_1) & -\frac{1}{2}\|W_{12}\|_2\\
-\frac{1}{2}\|W_{12}\|_2 & \lambda_m(W_2)
\end{bmatrix}.
\end{align*}
The conditions given at \refeqn{c2}, \refeqn{c1b}, \refeqn{kRkWb} guarantee that all of matrices $M_{11},M_{12},M_{21},M_{22},W$ are positive definite. %This implies that the zero equilibrium of the tracking error is stable in the sense of Lyapunov and the tracking error variables asymptotically converge to zero. Also, the estimation errors are uniformly bounded.

\section{Proof of Proposition \ref{prop:Pos2}}\label{sec:pfPos2}

According to the proof of Proposition \ref{prop:Att}, the attitude tracking errors asymptotically decrease to zero, and therefore, they enter the region given by \refeqn{Psi0} in a finite time $t^*$, after which the results of Proposition \ref{prop:Pos} can be applied to yield attractiveness. The remaining part of the proof is showing that the tracking error $z_1=[\|e_x\|,\|e_v\|]^T$ is bounded in $t\in[0, t^*]$. This is similar to the proof given at~\cite{LeeLeoAJC13}.

\end{document}